\documentclass{amsart}
\usepackage{amsmath,amssymb,amsthm,amsfonts,mathrsfs}
\usepackage{mathtools}
\usepackage{dsfont}
\usepackage{fullpage}
\usepackage{color}
\usepackage[colorlinks=true,
linkcolor=blue,
anchorcolor=blue,
citecolor=red
]{hyperref}
\allowdisplaybreaks
\usepackage{colonequals}
\usepackage{enumitem}
\usepackage{comment}

\newtheorem{theorem}{Theorem}[section]
\newtheorem{corollary}[theorem]{Corollary}
\newtheorem{lemma}[theorem]{Lemma}

\theoremstyle{definition}
\newtheorem{definition}[theorem]{Definition}
\newtheorem{example}[theorem]{Example}
\newtheorem{question}[theorem]{Question}
\theoremstyle{remark}
\newtheorem{remark}[theorem]{Remark}

\newcommand{\N}{\mathbb{N}}
\newcommand{\Z}{\mathbb{Z}}
\newcommand{\Q}{\mathbb{Q}}

\newcommand{\C}{\mathbb{C}}
\newcommand{\F}{\mathbb{F}}

\renewcommand{\P}{\mathbb{P}}

\newcommand{\cA}{\mathcal{A}}

\newcommand{\cG}{\mathcal{G}}

\newcommand{\cP}{\mathcal{P}}

\renewcommand{\mod}[1]{\,(\on{mod}#1)}
\newcommand{\of}[1]{\left(#1\right)}
\newcommand{\set}[1]{\left\{#1\right\}}

\newcommand{\norm}[1]{\left\Vert#1\right\Vert}

\newcommand{\smat}[1]{\begin{smallmatrix}#1\end{smallmatrix}}

\newcommand{\on}{\operatorname}
\renewcommand{\Re}{\on{Re}}

\newcommand{\Gal}{\on{Gal}}
\newcommand{\cI}{\mathcal{I}}

\newcommand{\GL}{{\rm{GL}}}
\newcommand{\RNum}[1]{\uppercase\expandafter{\romannumeral #1\relax}}

\newcommand{\rN}{\mathrm{N}}

\title[Alladi's formula]{Generalizations of Alladi's formula for arithmetical semigroups}
\author[L. Duan]{Lian Duan}
\address{Department of Mathematics, Colorado State University, Fort Collins, CO 80523, USA}
\email{lian.duan@colostate.edu}

\author[N. Ma]{Ning Ma}
\address{Department of Mathematics, State University of New York at Buffalo, Buffalo, NY 14260, USA}
\email{nma22@buffalo.edu}

\author[S. Yi]{Shaoyun Yi}
\address{Department of Mathematics, University of South Carolina, Columbia, SC 29208, USA}
\email{yishaoyun926@gmail.com}

\makeatletter
\@namedef{subjclassname@2020}{\textup{}2020 Mathematics Subject Classification}
\makeatother

\begin{document}
	\begin{abstract}
In this article, we prove that a general version of Alladi's formula with Dirichlet convolution holds for arithmetical semigroups satisfying Axiom $A$ or Axiom $A^{\#}$. As applications, we apply our main results to certain semigroups coming from algebraic number theory, arithmetical geometry and graph theory.
\end{abstract}
	
\maketitle

\section{Introduction and main results}\label{Sect. Intro}

Let $\mu: \Z_{\geq 1}\to \{0, \pm 1\}$ be  the M\"{o}bius function defined by $\mu(n)=(-1)^k$ if $n$ is the product of $k$ distinct primes, and zero otherwise.  
Let $p_{\min}(n)$ be the smallest prime factor of $n$ with $p_{\min}(1)=1$. In 1977, Alladi  \cite{Alladi1977} discovered a relationship between the M\"{o}bius function $\mu(n)$ and the density of primes in arithmetic progressions by showing  
\begin{equation}\label{alladi}
	-\sum_{\smat{n\geq 2\\ p_{\min}(n)\equiv \ell\, (\on{mod}k)}}\frac{\mu(n)}{n}=\frac1{\varphi(k)}
\end{equation}
for any positive integer $k$ and integer $\ell$ with $(\ell,k)=1$, where $\varphi$ is Euler's totient function. Recently, Wang \cite{Wang2020jnt} proved that if an arithmetic function $a:\Z_{\geq 1}\to \C$ satisfies $a(1)=1$ and $\sum_{n=2}^{\infty}\frac{|a(n)|}{n}\log\log n<\infty$, then 
\begin{equation}\label{Wang_convolution_alladi}
    -\sum_{\smat{n\geq 2\\ p_{\min}(n)\in S}}\frac{\mu * a(n)}{n}=\delta(S),
\end{equation}
where $\mu* a$ is the Dirichlet convolution of $\mu$ and $a$ and $\delta(S)$ is the natural density for a set $S$ of primes. In particular, if one takes $a$ to be such that $a(1)=1$ and $a(n)=0$ for all $n\in \Z_{\geq 2}$, then \eqref{Wang_convolution_alladi} recovers the result of \eqref{alladi} by the prime number theorem in arithmetic progressions in which we have $\delta(S)=1/\varphi(k)$. Notice that $\Z_{\geq 1}$ is a multiplicative semigroup generated by rational prime integers. Indeed, the semigroup structure satisfies the Axiom $A$ (see Sect.~\ref{Sect: Axiom_A}). This motivates us to consider the following question: 
\begin{question}
To what extend one can generalize the Alladi's formula and still get an analogue of \eqref{Wang_convolution_alladi}?
\end{question}

In this work, we discover proper formulations of Wang's result over all arithmetical semigroups satisfying Axiom $A$ or Axiom $A^{\#}$. More precisely, due to the fact that the ``smallest prime factor'' of an element in a general arithmetical semigroup $\cG$ is not well-defined, we focus on those arithmetic functions $a$ only supported on the \textit{distinguished set} $\mathfrak{D}(\cG,S)$ defined in \eqref{Definition of distingushable_Axiom A} and \eqref{Definition of distingushable_Axiom A_sharp}. However, introducing this distinguished set causes a new difficulty which cannot be handled by Wang's method in \cite{Wang2020jnt}. To deal with this difficulty, we use another method to avoid analyzing distinguished sets. By this trick, we manage to show a variation of Alladi's formula with convolution still hold for every such arithmetical semigroup; see Theorem~\ref{Thm: main_thm_convolution_Axiom A} and Theorem~\ref{Thm: main_thm_convolution_Axiom A_sharp} for details. Moreover, we apply our main results specifically to the arithmetical semigroups of interests, i.e., from algebraic number theory, algebraic geometry and graph theory; see Sect.~\ref{Sect: application_details} for explicit details. We hope our results and methods will find applications in future researches.

Recall \cite[Chap.~1, \S~1]{Knopfmacher1975} that an \emph{arithmetical semigroup} is a pair $(\cG, \|\cdot\|)$ (or simply $\cG$ if there is no confusion about $\|\cdot\|$). Here $\cG$ is a commutative semigroup with the identity element $e_{\cG}$, which is freely generated by a finite or countable subset $\cP$ of \emph{primes} of $\cG$. Explicitly, 
$$
\cG\colonequals\left\{\prod_{\text{finite product}} P_i^{a_i}\colon  P_i\in \cP \text{ and } a_i\geq 0 \text{ for all }i\right\}.
$$
In addition, there exists a discrete real-valued \emph{norm mapping} $\|\cdot\|$ on $\cG$ such that 
\begin{enumerate}
    \item[$(1)$] $\|e_{\cG}\|=1$, and $\|P\|>1$ for every $P\in \cP$. 
    \item[$(2)$] $\|gh\|=\|g\|\cdot \|h\|$ for all $g,h\in \cG$. 
    \item[$(3)$] For each real number $x>0$, the total number $\mathcal{N}(x)$ of elements $g\in\cG$ of norm $\norm{g}\leq x$ is finite. 
\end{enumerate}
Note that the conditions (1)-(3) are equivalent to conditions (1) and (2) together with
\begin{enumerate}[resume]
\item[$(3)^\prime$]For each real number $x>0$, the total number $\pi(x)$ of elements $P\in\cP$ of norm $\norm{P}\leq x$ is finite.
\end{enumerate}

Sometimes, instead of considering the norm mapping as above, it will be more convenient to consider a \emph{degree mapping} $\partial$ on $\cG$ (see \cite[Sect.~1.1]{KnopfmacherZhang2001}), which is defined as $\partial(g)\colonequals \log_c\norm{g}$ for some fixed constant $c>1$. So the conditions $(1)-(3)$ above read as 
\begin{enumerate}
    \item[($1^\#$)] $\partial(e_{\cG})=0$ and $\partial(P)>0$ for every $P\in \cP$. 
    \item[($2^\#$)] $\partial(gh)=\partial(g)+\partial(h)$ for all $g,h\in \cG$. 
    \item[($3^\#$)] For each real number $x>0$, the total number $\mathcal{N}^{\#}(x)$ of elements $g\in\cG$ with $\partial(g)\leq x$ is finite. 
\end{enumerate}
Among arithmetical semigroups, we are particularly interested in two types for which some additional axiom is satisfied. In the following subsections, each type together with the corresponding main result will be introduced. 

\subsection{Axiom \texorpdfstring{$A$}{A} type arithmetical semigroups}\label{Sect: Axiom_A}

\begin{definition}\cite[Chap.~4, \S~1]{Knopfmacher1975}\label{Def: Axiom_A}
    An arithmetical semigroup $\cG$ satisfies \emph{Axiom $A$} if 
    \begin{equation}\label{Eqn: Axiom_A}
        \mathcal{N}(x)=c_{\cG}x+O(x^\eta)
    \end{equation}
    with suitable constants $c_{\cG}>0$ and $0<\eta<1$\footnote{In \cite[Chap.~4, \S~1]{Knopfmacher1975}
an arithmetical semigroup $(\cG, \|\cdot\|)$ satisfies Axiom $A$ if 
    $
        \mathcal{N}(x)=c_{\cG}x^{\delta}+O(x^{\eta \delta}),
    $
   with suitable constants $c_{\cG}>0$ and $0<\eta<1$. In our work, for technical convenience we always normalize $\cG$ by taking the new norm $\norm{g}\colonequals |g|^{\delta}$ such that \eqref{Eqn: Axiom_A} holds.}.
\end{definition}

For any subset $S\subseteq\cP$, we say that $S$ has a \textit{natural density} $\delta(S)$ , if the following limit exists:
\begin{equation}
	\delta(S)\colonequals\lim_{x\to\infty}\frac{\pi_{\cG,S}(x)}{\pi_{\cG}(x)},
\end{equation}
where $\pi_{\cG,S}(x)\colonequals\#\set{P\in S:\|P\|\leq x}$ and $\pi_{\cG}(x)\colonequals\pi_{\cG,\cP}(x)$. For $g,h\in \cG$, we say $h| g$, if there exist $r\in \cG$ such that $hr=g$. Let $\rN_-(g)\colonequals\min\{\|P\|\colon P|g\}$ be the minimum norm of all prime factors of $g\in\cG$. And we say that $g$ is \textit{distinguishable} if $g\neq e_{\cG}$ and there is a unique prime factor, say $P_{\min}(g)$, of $g$ attaining the minimum norm $\rN_{-}(g)$. Furthermore, we define\footnote{{Note that our definition is a little different from \cite{Dawsey2017, SweetingWoo2019, KuralMcDonaldSah2020}} in the sense that they do not include the identity. }
\begin{equation}\label{Definition of distingushable_Axiom A}
	\mathfrak{D}(\cG,S)\colonequals\set{g\in \cG: g \text{ is distinguishable and } P_{\min}(g)\in S}\cup \{e_{\cG}\}.
\end{equation}

We will also need the following natural generalization of the \emph{M\"{o}bius function} to elements $g\in\mathcal{G}\colon$
$$
\mu_{\cG}(g)=
\begin{cases}
	1 & \text{ if } g=e_{\cG},\\
	(-1)^k & \text{ if }k>0 \text{ and } g=P_1\cdots P_k,\\
	0 & \text{ if } P^2|g \text{ for some prime } P.
\end{cases}
$$ 
Throughout, we will write $\mu(g)\colonequals\mu_{\cG}(g)$ for simplicity when the context is clear. An \emph{arithmetic function} $a:\cG\to \C$ is said to be \emph{supported on $\mathfrak{D}(\cG,S)$} if $a(g)=0$ unless $g\in\mathfrak{D}(\cG,S)$. And the \emph{Dirichlet convolution} of $\mu$ and $a$ is defined by $\mu*a(g)=\sum_{h| g}\mu(h)a(g/h)$. Moreover, we say that $a$ is the identity of the convolution ring of arithmetic functions over elements of $\cG$ if $a(g)=0$ for all $g\neq e_{\cG}$ and $a(e_{\cG})=1$. With all above notations, we can state our first main result of this work. 

\begin{theorem}\label{Thm: main_thm_convolution_Axiom A}
Given an Axiom $A$ type arithmetical semigroup $\cG$, assume that $S\subseteq\cP$ has a natural density $\delta(S)$. Suppose that $a\colon \mathcal{G}\to \C$ is an arithmetic function supported on $\mathfrak{D}(\cG, S)$ with $a(e_{\cG})=1$, and 
$$\lim_{x\to\infty}\sum_{\smat{2\le \norm{g}\le x\\ g\in\mathfrak{D}(\cG,S)}}\frac{|a(g)|}{\norm{g}}\log\log \norm{g}<\infty.$$

Let $\mu* a$ be the Dirichlet convolution of $\mu$ and $a$. Then
	\begin{equation}\label{Eqn: Convolution_mu_a for Axiom A}
		-\lim_{x\to\infty}  \sum_{\smat{2\le \norm{g}\le x\\ g\in \mathfrak{D}(\cG,S)}}\frac{\mu*a(g)}{\norm{g}}=\delta(S).
	\end{equation}	
In particular, if $a$ is the identity of the convolution ring of arithmetic functions over elements of $\cG$, then 
\begin{equation}\label{Eqn: Alladi_Axiom A}
	-\lim_{x\to\infty}  \sum_{\smat{2\le \norm{g}\le x\\ g\in \mathfrak{D}(\cG,S)}}\frac{\mu(g)}{\norm{g}}=\delta(S).
	\end{equation}

\end{theorem}
\begin{corollary}\label{Main Corollary_Axiom A} 
With $\cG, S$ and $\cP$ as in Theorem~\ref{Thm: main_thm_convolution_Axiom A}, suppose that $a\colon \cG\to \C$ is an arithmetic function supported on $\mathfrak{D}(\cG, S)$ with $a(e_{\cG})=1$ and $|a(g)|\ll \|g\|^{-\alpha}$ for some $\alpha>0$, then
\begin{equation}\label{mainthmeqvaphi_Axiom A}
		-\lim_{x\to\infty}  \sum_{\smat{2\le \norm{g}\le x\\ g\in \mathfrak{D}(\cG,S)}}\frac{\mu*a(g)}{\varphi(g)}=\delta(S),
	\end{equation}
where $\varphi(g)$ is Euler's totient function defined by $
\varphi(g)=\|g\|\prod_{P|g}\left(1-\frac{1}{\|P\|}\right).    
$	
\end{corollary}

\subsection{Axiom \texorpdfstring{$A^{\#}$}{A\#} type arithmetical semigroups}\label{Sect: Axiom_A_sharp}
In this case (unless otherwise stated), it shall be assumed that the degree map $\partial$ has image in $\Z_{\geq 0}$. Furthermore, we call $\cG$ together with $\partial$ an \emph{additive} arithmetical semigroup. In particular, we write
$$
\mathcal{G}\colonequals\left\{\sum_{\text{finite sum}} a_iP_i\colon P_i\in\mathcal{P} \text{ and } a_i\geq 0 \text{ for all }i\right\}.
$$
In the following let $G^{\#}(n)$ be the total number  of elements of degree $n$ in $\cG$. 
\begin{definition}\cite[Sect.~1.1]{KnopfmacherZhang2001}\label{Def: Axiom_A_sharp}
    An arithmetical semigroup $\cG$ satisfies \emph{Axiom A$^{\#}$} if
    \begin{equation}\label{Eqn: Axiom_A_sharp}
G^{\#}(n)=c_{\cG}q^{n}+O(q^{\eta n})
    \end{equation}
    with suitable constants $c_{\cG}>0$, $q>1$ and $0\leq \eta<1$.
\end{definition}
In addition, we assume that $Z_{\cG}(-q^{-1})\neq 0$.\footnote{This is a technical condition to simplify our arguments. In fact (see e.g. \cite{KnopfmacherZhang2001}) many consequences of Axiom $A^{\#}$ are unrelated to the value of $Z_{\cG}(-q^{-1})$, and so the simplifying additional condition would only sometimes become relevant. For the semigroups which do not satisfy this condition, one can refer to \cite{IM91}.} Here the function $Z_{\cG}(z)\colonequals\prod_{P}(1-z^{\partial(P)})^{-1}$ is the so-called zeta function of $\cG$.

Similarly but differently, for $\cG$ satisfying Axiom $A^{\#}$ as above, we define the \textit{natural density} $\delta(S)$ for a subset $S\subseteq \cP$ by the following limit if it exists:
\begin{equation}
	\delta(S)\colonequals\lim_{n\to\infty}\frac{\pi^{\#}_{\cG,S}(n)}{\pi^{\#}_{\cG}(n)},
\end{equation}
where $\pi^{\#}_{\cG,S}(n)\colonequals\#\set{P\in S: \partial(P)=n}$ and $\pi^{\#}_{\cG}(n)\colonequals\pi^{\#}_{\cG,\cP}(n)$. For $g,h\in \cG$, we say $g\ge h$, or equivalently, $h| g$, if there exists $r\in \cG$ such that $h+r=g$. Let $d_{-}(g)\colonequals\min\set{\partial( P): P| g}$ be the minimum degree of all prime factors of $g\in\cG$. We say that $g$ is \textit{distinguishable} if $g\neq e_{\cG}$ and there is a unique prime factor, say $P_{\min}(g)$, of $g$ attaining
the minimum degree $d_{-}(g)$. Furthermore, we define
\begin{equation}\label{Definition of distingushable_Axiom A_sharp}
	\mathfrak{D}^{\#}(\cG,S)\colonequals\set{g\in \mathcal{G}: g \text{ is distinguishable and } P_{\min}(g)\in S}\cup \{e_{\cG}\}.
\end{equation}
We will write $\mathfrak{D}(\cG,S)\colonequals\mathfrak{D}^{\#}(\cG,S)$ for simplicity when the context is clear. Moreover, the general M\"obius function $\mu_{\cG}$, arithmetic functions supported on $\mathfrak{D}(\cG,S)$ and the Dirichlet convolution of two arithmetic functions are defined in the same way as those in the previous subsection (perhaps except that one needs to translate the group law from multiplication to addition). In the following theorem, we fix the norm map $\|\cdot\|: \cG\to \C$ to be $\|g\|=q^{\partial(g)}$ with the base $q$ as in Definition~\ref{Def: Axiom_A_sharp}.

\begin{theorem}\label{Thm: main_thm_convolution_Axiom A_sharp}
Given an Axiom $A^{\#}$ type arithmetical semigroup $\cG$, assume that $S\subseteq\cP$ has a natural density $\delta(S)$. Suppose that $a\colon \cG\to \C$ is an arithmetic function supported on $\mathfrak{D}(\cG, S)$ with $a(e_{\cG})=1$, and 
$$\lim_{n\to\infty}\sum_{\smat{1\le \partial(g)\le n\\ g\in \mathfrak{D}(\cG,S)}}\frac{|a(g)|}{\|g\|}\log\log \|g\|<\infty.$$

Let $\mu* a$ be the Dirichlet convolution of $\mu$ and $a$. Then
	\begin{equation}\label{Eqn: Convolution_mu_a_Axiom A_sharp}
		-\lim_{n\to\infty}  \sum_{\smat{1\le \partial(g)\le n\\ g\in \mathfrak{D}(\cG,S)}}\frac{\mu*a(g)}{\|g\|}=\delta(S). 
	\end{equation}	
In particular, if $a$ is the identity of the convolution ring of arithmetic functions over elements of $\cG$, then 
\begin{equation}\label{Eqn: Alladi_Axiom A_sharp}
		-\lim_{n\to\infty}  \sum_{\smat{1\le \partial(g)\le n\\ g\in \mathfrak{D}(\cG,S)}}\frac{\mu(g)}{\|g\|}=\delta(S).
	\end{equation}
\end{theorem}

\begin{corollary}\label{Main Corollary_Axiom A_sharp} 
With $\cG, S$ and $\cP$ as in Theorem~\ref{Thm: main_thm_convolution_Axiom A_sharp}, suppose that $a\colon \cG\to \C$ is an arithmetic function supported on $\mathfrak{D}(\cG, S)$ with $a(e_{\cG})=1$ and $|a(g)|\ll \|g\|^{-\alpha}$ for some $\alpha>0$, then
\begin{equation}\label{mainthmeqvaphi}
		-\lim_{n\to\infty}  \sum_{\smat{1\le \partial(g)\le n\\ g\in \mathfrak{D}(\cG,S)}}\frac{\mu*a(g)}{\varphi(g)}=\delta(S),
	\end{equation}
where $\varphi(g)$ is Euler's totient function defined by $
\varphi(g)=\|g\|\prod_{P|g}\left(1-\frac{1}{\|P\|}\right).    
$	
\end{corollary}
\subsection{Prior work}
Alladi's original result \eqref{alladi} was generalized by Dawsey \cite{Dawsey2017} to the setting of Chebotarev densities for finite Galois extensions of $\Q$, and further by Sweeting and Woo \cite{SweetingWoo2019} to number fields. Later, Kural-McDonald-Sah \cite{KuralMcDonaldSah2020} generalized  all these results to natural densities of sets of prime ideals of a number field $K$. A recent work  \cite{DuanWangYi2020} of Wang with the first and third authors of this article showed the analogue of Kural-McDonald-Sah's result over global function fields. Wang \cite{Wang2020ijnt, Wang2020rj, Wang2020jnt} showed the analogues of these results over $\Q$ for some arithmetic functions other than $\mu$. It is worth to notice that in the partition setting (subject to Axiom $C$ \cite[Chap.~8, \S~1]{Knopfmacher1975}), a series work of Schneider \cite{Schneider2017, Schneider2020}, Ono-Wagner-Schneider \cite{OnoSchneiderWagner2017, OnoSchneiderWagner2020} and Dawsey-Just-Schneider \cite{DawseyJustSchneider2021} have built the framework of the corresponding generalizations. The interested readers are recommended to their papers.

\subsection{A quick summary of some applications}\label{Sect: Application_short}
There are many examples of arithmetical semigroups satisfying Axiom $A$ or Axiom $A^{\#}$; see \cite{Knopfmacher1975, KnopfmacherZhang2001, Terras-zeta-fun-graph}. As for a reference of other types, see \cite[Chap.~8, \S~1]{Knopfmacher1975}. In this subsection, we list the applications of Theorem~\ref{Thm: main_thm_convolution_Axiom A} and Theorem~\ref{Thm: main_thm_convolution_Axiom A_sharp} to some examples of interest in number theory, arithmetical geometry and graph theory. To save the space, only the essential information will be included in this subsection, and more detailed background of each example will be provided in Sect.~\ref{Sect: application_details}. In particular, we would like to fix the following notations. 
\begin{enumerate}
    \item For a number field $K$, we will denote the semigroup of integral ideals by $\mathcal{I}_K$. 
    \item For a $d$-dimensional smooth projective variety $X$ over a finite field $\F_q$, where $q=p^r$ is a power of prime integer $p$, we will denote the semigroup of effective $0$-cycles of $X$ by $\mathcal{A}_X$. 
    \item For a finite connected undirected graph $G$ without degree-$1$ vertices, we will denote the semigroup generated by classes of primitive paths by $\mathcal{A}_G$.
\end{enumerate}
In Table~\ref{table: examples}, we summarize the necessary information for each of above examples. Specifically, for each of the above examples, we list the type, the prime elements and the general elements of the associated semigroup. We also specify the degree map (if it makes sense) and the norm map those are involved in the corresponding result. As a support that each semigroup in this table satisfies either Axiom $A$ or Axiom $A^{\#}$, we list its corresponding function $\mathcal{N}$ or $G^{\#}$, which is equivalent to the associated \emph{prime number theorem} (PNT).
\begin{table}
\centering
\resizebox{\linewidth}{!}{%
\begin{tabular}{|c || c | c | c|}
 \hline
 Semigroup & $\mathcal{I}_K$ & $\mathcal{A}_X$& $\mathcal{A}_G$ \\ [0.5ex] \hline
  Type   & Axiom $A$ & Axiom $A^{\#}$ & Axiom $A^{\#}$ \\ \hline 
 Prime elements& prime ideals $\mathfrak{p}$ & prime $0$-cycles $P$ over $\F_q$ & equivalent classes $[P]$ of primitive paths \\ \hline
 General elements & integral ideals $\mathfrak{a}$ of $\mathcal{O}_K$ & effective $0$-cycles $A$ of $X$ over $\F_q$ &  finite formal sum $[C]=\sum a_i[P_i], a_i\geq 0$ \\ \hline
  Degree map $\partial$  &  & $\partial(P)=\#\{\text{geometric points above }P\}$ & $\partial([P])=\text{length of }P$\\ \hline
  Norm map $\|\cdot\|$ & $\| \mathfrak{p}\|=\# (\mathcal{O}_K/\mathfrak{p})$ & $\|P\|=(q^d)^{\partial(P)}$ & $\|[P]\|=1/R_G^{\partial([P])}$\\ \hline
 $\mathcal{N}$ or $G^{\#}$ & $\mathcal{N}(x)=c_{K}x+O(x^{1-1/[K:\Q]})$ & $G^{\#}(n)=c_{X}(q^d)^{\Delta_G n}+O(q^{(d-1/2)n})$ & $G^{\#}(n)=c_{G}/R_G^{\Delta_G n}+O(1/R_G^{\Delta_G\eta n})$\\ \hline
 PNT & $\pi_{K}(x)=\dfrac{x}{\log x}+O(x\,\mathrm{exp}(-c_K\sqrt{\log x})) $& $\pi^{\#}_X(n)=\dfrac{(q^d)^{n}}{n}+O\left(\dfrac{q^{(d-1/2)n}}{n}\right)$ & $\pi^{\#}_G(n)= \dfrac{1}{n R^{\Delta_G n}_G}+O\left(\dfrac{1}{n R_G^{\Delta_G \eta n}}\right)$ \\ \hline
 Main result & Corollary~\ref{Cor: Alladi_number_field}&Corollary~\ref{Cor: Alladi_varieties}&Corollary~\ref{Cor: Alladi_graph}\\ \hline
 Reference section & Sect.~\ref{Sect: applicatioin_num_field}& Sect.~\ref{Sect: application_0_cycles}& Sect.~\ref{Sect: application_graph_theory} \\ \hline

\end{tabular}}
\caption{Examples of semigroups}
\label{table: examples}
\end{table}

It follows from Table~\ref{table: examples} and Theorem~\ref{Thm: main_thm_convolution_Axiom A} that we have the following result for the semigroup $\mathcal{I}_K$ of integral ideals of a number field $K$.
\begin{corollary}\label{Cor: Alladi_number_field}
For a number field $K$, assume that $S\subseteq \cP$ is a subset of prime ideals with natural density $\delta(S)$. For any arithmetic function $a: \mathcal{I}_K\to \C$ supported on $\mathfrak{D}(\mathcal{I}_K, S)$ with $a(\mathcal{O}_K)=1$, and 
$$\lim_{x\to\infty}\sum_{\smat{2\le \norm{\mathfrak{a}}\le x\\ \mathfrak{a}\in\mathfrak{D}(\cI_K,S)}}\frac{|a(\mathfrak{a})|}{\norm{\mathfrak{a}}}\log\log \norm{\mathfrak{a}}<\infty.$$
Then
	\begin{equation}\label{Eqn: Main result for number fields K as a cor of Axiom A}
		-\lim_{x\to\infty}  \sum_{\smat{2\le \norm{\mathfrak{a}}\le x\\ \mathfrak{a}\in \mathfrak{D}(\cI_K,S)}}\frac{\mu*a(\mathfrak{a})}{\norm{\mathfrak{a}}}=\delta(S).
	\end{equation}	
In particular, we have
\begin{equation}\label{Recover KMS 2020}
	-\lim_{x\to\infty}  \sum_{\smat{2\le \norm{\mathfrak{a}}\le x\\ \mathfrak{a}\in \mathfrak{D}(\cI_K,S)}}\frac{\mu(\mathfrak{a})}{\norm{\mathfrak{a}}}=\delta(S).
	\end{equation}
\end{corollary}

\begin{remark}
Note that \eqref{Recover KMS 2020} recovers the main theorem in \cite{KuralMcDonaldSah2020}. In this sense, Corollary~\ref{Cor: Alladi_number_field} is a generalization of the result in \cite{KuralMcDonaldSah2020} to some arithmetic functions other than $\mu$. On the other hand, if we let $K=\Q$, then \eqref{Eqn: Main result for number fields K as a cor of Axiom A} in this case is exactly the main theorem of \cite{Wang2020jnt}. And so in this sense, Corollary~\ref{Cor: Alladi_number_field} generalizes the result in \cite{Wang2020jnt} to all number fields.
\end{remark}

It follows from Table~\ref{table: examples} and Theorem~\ref{Thm: main_thm_convolution_Axiom A_sharp} that we have the following result for the semigroup $\cA_X$ of effective $0$-cycles of a $d$-dimensional smooth projective variety $X$ over a finite field $\F_q$.
\begin{corollary}\label{Cor: Alladi_varieties}
For a smooth projective $d$-dimensional variety $X$ defined over a finite field $\F_q$, assume that $S\subseteq \cP$ is a subset of prime $0$-cycles with natural density $\delta(S)$. For any arithmetic function $a: \cA_{X}\to \C$ supported on $\mathfrak{D}(\cA_X, S)$ with $a(0)=1$, and 
$$\lim_{n\to\infty}\sum_{\smat{1\le \partial(A)\le n\\ A\in \mathfrak{D}(\cA_X,S)}}\frac{|a(A)|}{\|A\|}\log\log \|A\|<\infty.$$
Then
	\begin{equation}
		-\lim_{n\to\infty}  \sum_{\smat{1\le \partial(A)\le n\\ A\in \mathfrak{D}(\cA_X,S)}}\frac{\mu*a(A)}{\|A\|}=\delta(S). 
	\end{equation}	
In particular, we have
\begin{equation}\label{Eqn: recover_DWY}
		-\lim_{n\to\infty}  \sum_{\smat{1\le \partial(A)\le n\\ A\in \mathfrak{D}(\cA_{X},S)}}\frac{\mu(A)}{\|A\|}=\delta(S).
	\end{equation}
\end{corollary}
\begin{remark}
When $X$ has the dimension one, we have $d=1$ in the above setups. Then one can check that in this case the semigroup $\cA_X$ coincides with the semigroup of effective divisors of $X$. It follows that \eqref{Eqn: recover_DWY} recovers the main theorem of \cite{DuanWangYi2020}. Hence one can think Corollary~\ref{Cor: Alladi_varieties} as a generalization of the result in \cite{DuanWangYi2020} to some arithmetic functions other than $\mu$ for all higher dimensional varieties. 
\end{remark}

\begin{remark}
We note that the ``smooth and projective'' condition in Corollary~\ref{Cor: Alladi_varieties} is not essential. In fact, if $X$ is a geometrically connected scheme, then $0$-cycles of $X$ are still well defined and so analogous results in Table~\ref{table: examples} can be deduced from its Hasse-Weil zeta function by a similar argument as in Sect.~\ref{Sect: application_0_cycles}; see \cite{Chen2017} for more details about the zeta function of $X$ and the prime number theorem in general cases.
\end{remark}

\begin{remark}
Based on Corollary~\ref{Cor: Alladi_varieties}, we can also deduce some other consequences regarding to the distribution of hyperplanes of projective spaces; see Sect.~\ref{Sect: application_0_cycles} for more details.
\end{remark}

Similarly, it follows from Table~\ref{table: examples} and Theorem~\ref{Thm: main_thm_convolution_Axiom A_sharp} that we have the following result for the semigroup $\cA_G$ generated by classes of primitive paths of a finite connected undirected graph $G$ without degree-$1$ vertices.

\begin{corollary}\label{Cor: Alladi_graph}
For a finite connected undirected graph $G$ without degree-$1$ vertices, assume that $S\subseteq \cP$ is a subset of classes of primitive paths with natural density $\delta(S)$. For any arithmetic function $a: \cA_{G}\to \C$ supported on $\mathfrak{D}(\cA_G, S)$ with $a([0])=1$, and 
$$\lim_{n\to\infty}\sum_{\smat{1\le \partial([C])\le n\\ [C]\in \mathfrak{D}(\cA_G,S)}}\frac{|a([C])|}{\|[C]\|}\log\log \|[C]\|<\infty.$$
Then
	\begin{equation}
		-\lim_{n\to\infty}  \sum_{\smat{1\le \partial([C])\le n\\ [C]\in \mathfrak{D}(\cA_G,S)}}\frac{\mu*a([C])}{\|[C]\|}=\delta(S). 
	\end{equation}	
In particular, we have
\begin{equation}
		-\lim_{n\to\infty}  \sum_{\smat{1\le \partial([C])\le n\\ [C]\in \mathfrak{D}(\cA_{G},S)}}\frac{\mu([C])}{\|[C]\|}=\delta(S).
	\end{equation}
\end{corollary}

\subsection{Organization of this paper}
We will prove Theorem~\ref{Thm: main_thm_convolution_Axiom A_sharp} for Axiom $A^{\#}$ type arithmetical semigroup and Theorem~\ref{Thm: main_thm_convolution_Axiom A} for Axiom $A$ type arithmetical semigroup in Sect.~\ref{Sect: Proofs for semigps Axiom A_sharp} and Sect.~\ref{Sect: Free abe. semigp Axiom A}, respectively. Explicitly, to prove Theorem~\ref{Thm: main_thm_convolution_Axiom A_sharp}, we first give an analogous equidistribution result of the largest (in terms of degree $\partial(\cdot)$) prime factors of elements with the same degree in $\cG$ (Lemma~\ref{equidistribution_QSg_Axiom A_sharp}) and an analogue of Alladi's duality identity to $\cG$ (Lemma~\ref{duality_axiom A_sharp}). Next, we prove the intermediate theorem (Theorem~\ref{keylemma12combined}) and then Theorem~\ref{Thm: main_thm_convolution_Axiom A_sharp}. Due to the fact that the proofs of the results for Axiom $A$ type case use a similar strategy as that for Axiom $A^{\#}$ type case, we will sketch its proofs and only emphasis on key steps in Sect.~\ref{Sect: Free abe. semigp Axiom A}. Finally, in Sect.~\ref{Sect: application_details}, we explain the contents in Table~\ref{table: examples} in details, and also provide some more examples.  

\section{Proofs of main results for Axiom \texorpdfstring{$A^{\#}$}{A\#} type arithmetical semigroups}\label{Sect: Proofs for semigps Axiom A_sharp}
In this section, we discuss the situation of Axiom $A^{\#}$ type arithmetical semigroups. More precisely, there is a key result in this section, Theorem~\ref{keylemma12combined}, and then Theorem~\ref{Thm: main_thm_convolution_Axiom A_sharp} follows as an immediate consequence. 
Throughout this section, we fix $\cG$ to be an Axiom $A^{\#}$ type arithmetical semigroup, and adopt the definitions and notations as in Sect.~\ref{Sect: Axiom_A_sharp}. 

For any element $g\in \cG$, we take $d^+(g)\colonequals\max\set{ \partial(P): P|g}$ to be the largest degree of all prime factors of $g$, and $d^+(e_{\cG})=0$. Let
\begin{equation}\label{Defn of QSg for Axiom A_sharp}
    Q_S^{\#}(g)\colonequals\#\set{P\in S: \partial(P)=d^+(g), P|g}
\end{equation}
be the number of prime factors of $g$ in $S$ attaining the maximal degree $d^+(g)$. We will write $Q_S(g)\colonequals Q_S^{\#}(g)$ for simplicity when the context is clear. Then, with a similar process as in the proof of \cite[Theorem~4.4]{DuanWangYi2020} one can have the analogous asymptotic estimate for $Q_S(g)$ as follows.
\begin{lemma}\label{equidistribution_QSg_Axiom A_sharp} 
    If $S\subseteq \mathcal{P}$ has a natural density $\delta(S)$, then
    \begin{equation}\label{Eqn: equidistribution_QSg_Axiom A_sharp}
		\sum_{ \partial(g)=n}Q_S(g)=c_{\cG}\delta(S)q^{n}+o(q^{{n}}).
	\end{equation}
\end{lemma}
We will also require an analogue of Alladi's duality property \cite[Lemma~1]{Alladi1977} to $\cG$ for our purpose.
\begin{lemma}\label{duality_axiom A_sharp}
	Suppose that $f:\N\to\C$ is an arithmetic function with $f(0)=0$. Then for any $g\in \cG$ we have
	\begin{equation}\label{Eqn: duality_axiom A_sharp}
		\sum_{h|g}\mu(h)1_{\mathfrak{D}(\cG, S)}(h)f(d_{-}(h))=-Q_S(g)f(d^+(g)).
	\end{equation}
	Here $1_{\mathfrak{D}(\cG,S)}$ is the indicator function on $\mathfrak{D}(\cG,S)$ defined as in \eqref{Definition of distingushable_Axiom A_sharp}.
\end{lemma}
\begin{proof}
This directly generalizes \cite[Lemma~3.1]{DuanWangYi2020} and the desired result follows from a similar argument.
\end{proof}
Moreover, we will also need the following intermediate theorem.
\begin{theorem}\label{keylemma12combined}
	Let $f$ be any bounded arithmetic function with $f(0)=0$. Let $a$ be an arithmetic function as in Theorem~\ref{Thm: main_thm_convolution_Axiom A_sharp}. The followings are equivalent:
	\begin{equation}\label{Eqn: equidis of largest prime factors}
		\sum_{ \partial(g)=n}Q_S(g)f(d^+(g))\sim c_{\cG}\delta(S)q^{n}.
	\end{equation}	
	\begin{equation}\label{Eqn: analogue of main result in global function feild}
		-\lim_{n\to\infty}  \sum_{\smat{1\le \partial(g)\le n\\ g\in \mathfrak{D}(\cG,S)}}\frac{\mu(g)f(d_{-}(g))}{\|g\|}=\delta(S).
	\end{equation}
\begin{equation}\label{Eqn: to write into partial sum}
		-\lim_{n\to\infty}  \sum_{\smat{1\le \partial(g)\le n\\ g\in \mathfrak{D}(\cG,S)}}\frac{\mu*a(g)f(d_{-}(g))}{\|g\|}=\delta(S).
	\end{equation}
\end{theorem}

Before going into the details of proof of Theorem~\ref{keylemma12combined}, we would like to give the proof of Theorem~\ref{Thm: main_thm_convolution_Axiom A_sharp} based on this theorem. 

\begin{proof}[Proof of Theorem~\ref{Thm: main_thm_convolution_Axiom A_sharp}]
The desired result \eqref{Eqn: Convolution_mu_a_Axiom A_sharp} follows immediately from Lemma~\ref{equidistribution_QSg_Axiom A_sharp} and Theorem~\ref{keylemma12combined} by taking $f(n)=1$ for all $n\in\Z_{\geq 1}$.
\end{proof}
We will separate the proof of Theorem~\ref{keylemma12combined} into two parts: ``$\eqref{Eqn: equidis of largest prime factors}\Leftrightarrow \eqref{Eqn: analogue of main result in global function feild}$" and ``$\eqref{Eqn: analogue of main result in global function feild}\Leftrightarrow \eqref{Eqn: to write into partial sum}$". 
\subsection{Proof of \texorpdfstring{``$\eqref{Eqn: equidis of largest prime factors}\Leftrightarrow \eqref{Eqn: analogue of main result in global function feild}$"}{}}\label{Sect. for Proof of the first Key Theorem} First, we prove the following three lemmas which estimate several partial sums involving the general Möbius function $\mu$.

\begin{lemma}\label{Lemma of improving Cha2017}
Let $\cG$ be an Axiom $A^{\#}$ type arithmetical semigroup. Then for some $0\leq \eta<1$ we have
 \begin{equation}\label{Eqn: finite_sum_of_mu}
     \sum_{\partial(g)=n}\mu(g)=O(q^{\eta n}).
 \end{equation}
\end{lemma}
\begin{proof}
It is clear that for $\Re(s)>1$ we have
$$
\frac{1}{\zeta_{\cG}(s)}=\sum_{g\in\cG}\frac{\mu(g)}{\|g\|^s}=\sum_{\partial(g)\geq 0} \frac{\mu(g)}{q^{\partial(g)s}}.
$$
By taking $T=q^{-s}$, we can rewrite this series as
$$Z_{\cG}(T)=\sum_{n=0}^{\infty}C_{\mu}(n)T^n,$$
where $C_{\mu}(n)=\sum_{\partial(g)=n}\mu(g)$. Consider the following integral
\begin{equation}
    \frac{1}{2\pi i}\int_{|T|=q^{-\eta }}\frac{1}{Z_{\cG}(T)}\frac{dT}{T^{(n+1)}}=C_{\mu}(n),
\end{equation}
where we can choose $0\leq \eta<1$ such that $Z_{\cG}(T)$ has no zeros in the closed disk $|T|\le q^{-\eta}$. Then using a similar argument as in the proof of \cite[Theorem~1]{IM91}, we have that
\begin{equation}
    C_{\mu}(n)=\sum_{\partial(g)=n}\mu(g)=O(q^{\eta n}).
\end{equation}
\end{proof}
\begin{lemma}\label{Corollary for uniform bound of M(n, m)}
For any $m, n\in\Z_{\geq 0}$, define
\begin{equation}
    C(n,m)\colonequals\sum_{\smat{\partial(g)=n\\d_{-}(g)>m}}\mu(g).
\end{equation}
Then there exists some $0\leq \eta<1$ such that
\begin{equation}\label{Eqn: Bound for C(n, m)_Axiom A_sharp}
    C(n,m)\ll q^{\eta n}\mathrm{exp}(q^m).
\end{equation}
In particular, we have 
\begin{equation}\label{Eqn: Bound for M(n, m)_Axiom A_sharp}
    M(n,m)\colonequals\sum_{\smat{\partial(g)\le n\\d_{-}(g)>m}}\mu(g)\ll q^{\eta n}\mathrm{exp}(q^m).
\end{equation}
\end{lemma}
\begin{proof}
As in the proof of Lemma~\ref{Lemma of improving Cha2017}, if we consider the function 
\begin{equation}
\frac{1}{\zeta_{\cG}(s)}\prod_{\partial(P)\leq m}(1-\|P\|)^{-s})^{-1}=\sum_{\smat{g\in\cG\\d_{-}(g)>m}}\frac{\mu(g)}{\|g\|^s},  
\end{equation}
then one can see that
	\begin{equation}
		C(n,m)\ll q^{\eta n}\prod_{\partial(P)\le m}\of{1-{\|P\|^{-\eta}}}^{-1}
	\end{equation}
	uniformly for $n,m\ge1$. By the abstract prime number theorem (e.g. see \cite[Chap.~8]{Knopfmacher1979}, \cite{Cohen1989}), for the product over $\partial(P)\le m$ we have that
	$$-\sum_{\partial(P)\le m}\log (1-\|P\|^{-\eta})\ll \sum_{\partial(P)\le m}\|P\|^{-\eta}\ll\sum_{\partial(P)\le m}1\ll\sum_{k\le m} q^{k}/k
	\ll q^{m}.$$
Thus, the desired result \eqref{Eqn: Bound for C(n, m)_Axiom A_sharp} follows. As for the estimate \eqref{Eqn: Bound for M(n, m)_Axiom A_sharp}, we have
$$M(n, m)=\sum_{\ell\leq n}C(\ell, m)\ll \mathrm{exp}(q^m)\sum_{\ell \leq n}q^{\eta \ell}=q^{\eta n}\mathrm{exp}(q^m)\sum_{\ell \leq n}q^{\eta (\ell-n)}\ll q^{\eta n}\mathrm{exp}(q^m).$$
\end{proof}
\begin{lemma}\label{averagemu}
	For any bounded function $f$,	we have
	\begin{equation}\label{averagemueq}
		\sum_{\partial(g)\le n}\mu(g)1_{\mathfrak{D}(\cG,S)}(g)f(d_{-}(g))=O_{f,\cG}\of{\frac{q^{n}}{\log\log n}}.
	\end{equation}
	
\end{lemma}

\begin{proof} 
First, we break up the sum based on the degree of the minimal prime factor $P_{\min}(g)$ of $g\in \mathfrak{D}(\cG,S)$.
	\begin{align}\label{pfmusetup}
		\sum_{\partial(g)\le n}\mu(g)1_{\mathfrak{D}(\cG,S)}(g)f(d_{-}(g))&=\sum_{\smat{\partial(P)\le n\\ P\in S}}f(\partial(P))\sum_{\smat{\partial(g)\le n\\ P_{\min}(g)=P}}\mu(g)\nonumber\\
		&=\sum_{\smat{\partial(P)\le m\\ P\in S}}f(\partial(P))\sum_{\smat{\partial(g)\le n\\
				P_{\min}(g)=P}}\mu(g)\nonumber\\
		&\qquad+ \sum_{\smat{m<\partial(P)\le n\\ P\in S}}f(\partial(P))\sum_{\smat{\partial(g)\le n\\
				P_{\min}(g)=P}}\mu(g)\nonumber\\
		&\colonequals S_1+S_2,
	\end{align}
	where $m$ is to be chosen later. For the inside sum, observe that
	\begin{equation}\label{Eqn: torewriteasPsi}
	    \sum_{\smat{\partial(g)\le n\\P_{\min}(g)=P}}\mu(g)=-\sum_{\smat{\partial(g)-\partial(P)\le n\\d_{-}(g)>\partial(P)}}\mu(g).
	\end{equation}
Thus, by \eqref{Eqn: torewriteasPsi}, Lemma~\ref{Corollary for uniform bound of M(n, m)} and the abstract prime number theorem for $\cG$, for $S_1$ we have
	\begin{align}\label{s6estimate}
		S_1&=-\sum_{\smat{\partial(P)\le m\\ P\in S}}f(\partial(P))M(n-\partial(P), \partial(P))\\
		&\ll \sum_{\partial(P)\le m}|M(n-\partial(P), \partial(P))|\nonumber\\
		&\ll\sum_{\partial(P)\le m} q^{\eta(n-\partial(P))}\exp\of{q^{\partial(P)}}\nonumber\\
		&\ll q^{\eta n}\sum_{1\le k\le m}\frac{q^{k}}{k}\cdot q^{-\eta k}\exp\of{q^{k}}\nonumber\\
		&\ll q^{\eta n+(1-\eta)m}\exp\of{q^{m}}/m.
	\end{align}	
	For $S_2$, we have
	\begin{equation}\label{s7-1}
		S_2\ll \sum_{m<\partial(P)\le n}\sum_{\smat{\partial(g)\le n\\
				P_{\min}(g)=P}}1\le \Phi(n,m),
	\end{equation}
where $\Phi(n,m)\colonequals\sum_{\smat{\partial(g)\le n\\d_{-}(g)>m}}1$ for $m,n\ge1$. Moreover, by the sieve of Eratosthenes, we have
	\begin{equation}
		\Phi(n,m)=c_{\cG}q^{n}\prod_{\partial(P)\le m}\of{1-\|P\|^{-1} }+O\of{n\,2^{c_{\cG}q^{m}}}.
	\end{equation}
Again by the abstract prime number theorem for $\cG$, with a standard computation we obtain that
	\begin{equation}\label{s7-2}
		\Phi(n,m)\ll\frac{q^{n}}{m}+n\,2^{c_{\cG}q^{m}}.
	\end{equation}
	Taking  $m=[\log\log n]$ and combining \eqref{pfmusetup}, \eqref{s6estimate}, \eqref{s7-1}, and \eqref{s7-2} together, the desired estimate (\ref{averagemueq}) follows.
\end{proof}

Now we are ready to prove ``$\eqref{Eqn: equidis of largest prime factors}\Leftrightarrow \eqref{Eqn: analogue of main result in global function feild}$".
\begin{proof}[Proof of ``$\eqref{Eqn: equidis of largest prime factors}\Leftrightarrow \eqref{Eqn: analogue of main result in global function feild}$'']\label{Proof of Theorem keylemma}
	By Lemma~\ref{duality_axiom A_sharp}, we	have
	\begin{align}
		-\delta(S)c_{\cG}q^{n}&\sim -\sum_{\partial(g)=n}Q_S(g)f(d^+(g))\nonumber\\
		&=\sum_{ \partial(g)=n}\sum_{g\ge h}\mu(h)1_{\mathfrak{D}(\cG,S)}(h)f(d_{-}(h))\nonumber\\
		&=\sum_{\partial(h)\le n} \mu(h)1_{\mathfrak{D}(\cG,S)}(h)f(d_{-}(h)) \sum_{\partial(r)=n-\partial(h)}1\nonumber\\
		&=\sum_{\partial(h)\le n} \mu(h)1_{\mathfrak{D}(\cG,S)}(h)f(d_{-}(h)) \left(c_{\cG}\frac{q^{n}}{\|h\|}+O(q^{\eta (n-\partial(h))})\right) \nonumber\\
		&=\sum_{\partial(h)\le n} \frac{\mu(h)1_{\mathfrak{D}(\cG,S)}(h)f(d_{-}(h))}{\|h\|} c_{\cG}q^{n}+\sum_{\partial(h)\le n} \mu(h)1_{\mathfrak{D}(\cG,S)}(h)f(d_{-}(h))O(q^{\eta (n-\partial(h))}) \nonumber\\
		&\colonequals S_3+S_4.
	\end{align}
To prove ``$\eqref{Eqn: equidis of largest prime factors}\Leftrightarrow \eqref{Eqn: analogue of main result in global function feild}$", it suffices to show that $S_4=o(q^n)$. In fact, for $S_4$ we have
\begin{align*}
    S_4&=\sum_{\partial(h)\le n} \mu(h)1_{\mathfrak{D}(\cG,S)}(h)f(d_{-}(h))O(q^{\eta (n-\partial(h))})\\
    &=\sum_{\partial(h)\le n-N_0} \mu(h)1_{\mathfrak{D}(\cG,S)}(h)f(d_{-}(h))O(q^{\eta (n-\partial(h))})+\sum_{N_0<\partial(h)\le n} \mu(h)1_{\mathfrak{D}(\cG,S)}(h)f(d_{-}(h))O(q^{\eta (n-\partial(h))})\\
    &\colonequals S_5+S_6.
\end{align*}
It follows from Lemma~\ref{averagemu} and the straightforward calculations by taking $N_0=(\eta+\varepsilon)n<n$ for arbitrary $\varepsilon>0$ (for example, we can take $\varepsilon=(1-\eta)/2>0$) that $S_5=o(q^n)$ and $S_6=o(q^n)$ as desired.
\end{proof}
\subsection{Proof of \texorpdfstring{``$\eqref{Eqn: analogue of main result in global function feild}\Leftrightarrow \eqref{Eqn: to write into partial sum}$"}{}}\label{Sect. for Proof of the second Key Theorem} Our proof relies on the estimate of $R(n,m)$ defined by
\begin{equation}\label{R(x y)}
    R(n, m)\colonequals\sum_{\smat{0\leq \partial(g)\le n\\d_{-}(g)>m}}\frac{\mu(g)}{\|g\|}, 
\end{equation}
where $m,n$ are a pair of non-negative integers. We set $P_{\min}(e_{\cG})=\infty$ for convenience. First, we give an elementary bound for $R(n, m)$ in the following lemma. 

\begin{lemma}\label{lemma 1 for abs. value. of R(x y)}
For any $n, m\geq 0$, we have that
\begin{equation}\label{Eqn: elementrary_bound_for_R(n,m)}
R(n, m)=O(1).
\end{equation}
\end{lemma}

\begin{proof}
The approach is similar to that in \cite{Tao2010}. For $m\in\Z_{\geq 0}$, define $\mathcal{P}_m\colonequals\{P\mid \partial(P)>m\}$ and $\mathcal{P}^c_m\colonequals\{P\mid \partial(P)\leq m\}$. Let $\langle\mathcal{P}_m\rangle$ and $\langle\mathcal{P}^c_m\rangle$ be monoids generated by $\mathcal{P}_m$ and $\mathcal{P}^c_m$, respectively. It is easy to check that
\begin{equation}\label{Eqn: useful identity}
    1_{\langle\mathcal{P}^c_m\rangle}(g)=\sum_{\substack{r| g\\ r\in\langle \mathcal{P}_m\rangle}}\mu(r).
\end{equation}
Next, we consider the elements $g$ of degree $n$ and by \eqref{Eqn: useful identity} we have that
\begin{align*}
    \sum_{\substack{\partial(g)=n\\g\in\langle\mathcal{P}^c_m\rangle}}1=&\sum_{\partial(g)=n}\sum_{\substack{r| g\\ r\in\langle \mathcal{P}_m\rangle}}\mu(r)\\
    =&\sum_{\substack{\partial(r)\leq n\\ r\in\langle \mathcal{P}_m\rangle}}\mu(r)\sum_{\partial(h)=n-\partial(r)}1\\
=&\sum_{\substack{\partial(r)\leq n\\r\in\langle \mathcal{P}_m\rangle}}\mu(r)\left(c_{\cG}q^{n-\partial(r)}+c_0q^{\eta(n-\partial(r))}\right)\\
    =&c_{\cG}q^{n}\sum_{\substack{\partial(r)\leq n\\r\in\langle \mathcal{P}_m\rangle}}\frac{\mu(r)}{\|r\|}+c_0q^{\eta n}\sum_{\substack{\partial(r)\leq n\\r\in\langle \mathcal{P}_m\rangle}}\frac{\mu(r)}{\|r\|^{\eta}},
\end{align*}
where $c_0$ is some constant number. It follows that
\begin{align*}
    \Big |\sum_{\substack{\partial(r)\leq n\\r\in\langle \mathcal{P}_m\rangle}}\frac{\mu(r)}{\|r\|}\Big |\ll&\frac{1}{q^{n}}\sum_{\partial(g)=n}1+q^{(\eta-1)n}\sum_{\partial(r)\leq n}\frac{1}{\|r\|^{\eta}}\\
    =&c_{\cG}+O(q^{(\eta-1)n})+q^{(\eta-1)n}\sum_{\partial(r)\leq n}\frac{1}{\|r\|^{\eta}}.\\
    \ll&q^{(\eta-1)n}\sum_{k=0}^n\sum_{\partial(r)= k}\frac{1}{\|r\|^{\eta}}\\
    =&q^{(\eta-1)n}\sum_{k=0}^n\frac{1}{q^{\eta k}}\left(c_{\cG}q^{k}+O(q^{\eta k})\right)\\
\ll&1.
\end{align*}
Thus, this proves the lemma.
\end{proof}

In addition to the elementary bound for $R(n,m)$ as above, we will also need a refined estimate of $R(n,m)$ when $0\leq m\leq \log \log n$ for our purpose. More precisely, we will show the following lemma.
\begin{lemma}\label{lemma 2 for R(x y)}
For $0\leq m\leq \log\log n$ and some constant $0\leq \eta<1$ only depending on $\cG$, we have that
\begin{equation}\label{equation for est. R(x, y)}
R(n, m)\ll q^{(\eta -1+\varepsilon)n},
\end{equation}
where $\varepsilon$ is an arbitrary positive number.
\end{lemma}
\begin{proof}
We will separate two cases to discuss, i.e., $m=0$ and $m\neq 0$. If $m=0$, then we have
$$R(n, 0)=\sum_{\smat{\partial(g)\le n}}\frac{\mu(g)}{\|g\|}.$$
Assuming that $n\geq 1$, by the principle of inclusion-exclusion and a similar computation as in the proof of Lemma~\ref{lemma 1 for abs. value. of R(x y)}, we can get
\begin{align*}
    0=\sum_{\partial(g)=n}\sum_{r| g}\mu(r)=
    c_{\cG}q^{n}\sum_{\partial(r)\leq n}\frac{\mu(r)}{\|r\|}+c_0q^{\eta n}\sum_{\partial(r)\leq n}\frac{\mu(r)}{\|r\|^{\eta}} 
\end{align*}
for some constant $c_0$. It follows that
\begin{equation}\label{eqn_Rnm_refined_m0}
    \sum_{\partial(r)\leq n}\frac{\mu(r)}{\|r\|}=-\frac{c_0}{c_{\cG}}q^{(\eta -1)n}\sum_{k=0}^n\sum_{\partial(r)=k}\frac{\mu(r)}{\|r\|^{\eta}}.
\end{equation}
Consider $k\geq 1$ and by Lemma~\ref{Lemma of improving Cha2017}, then the inside summation
\begin{equation}\label{eqn_Rnm_refined_m0_with_eta}
        \sum_{\partial(r)=k}\frac{\mu(r)}{\|r\|^{\eta}}=\frac{1}{q^{\eta k}}\sum_{\partial(r)=k}\mu(r)    \ll 1.
\end{equation}
It is clear that \eqref{eqn_Rnm_refined_m0_with_eta} also holds for $k=0$. Thus, we have 
\begin{align*}
    \sum_{k=0}^n\sum_{\partial(r)=k}\frac{\mu(r)}{\|r\|^{\eta}}\ll n.
\end{align*}
Therefore, by \eqref{eqn_Rnm_refined_m0} and \eqref{eqn_Rnm_refined_m0_with_eta} we obtain that
\begin{equation}
        \sum_{\partial(r)\leq n}\frac{\mu(r)}{\|r\|}=O(q^{(\eta -1+\varepsilon)n}),
\end{equation}
where $\varepsilon$ is an arbitrary positive number. This proves the desired \eqref{equation for est. R(x, y)} for the special case $m=0$. 

Now suppose that $1\leq m \leq \log\log n$, by Lemma~\ref{Corollary for uniform bound of M(n, m)} we have that
\begin{equation}\label{Key est for M(n, m)}
    M(n,m)\ll q^{\eta n}.
\end{equation}
It follows from \cite[Proposition~(1.2.1)]{KnopfmacherZhang2001} that $\zeta_{\cG}(s)$ has a simple pole at $s=1$. And so we have that
\begin{equation}\label{Eqn: simple pole gives the desired 0 equation}
    \sum_{\partial(g)\geq 0}\frac{\mu(g)}{\|g\|}=\lim_{s\to 1}\frac{1}{\zeta_{\cG}(s)}=0.
\end{equation}
Let $\mathcal{P}^c_m=\{P\mid \partial(P)\leq m\}$ as in the proof of Lemma~\ref{lemma 1 for abs. value. of R(x y)}. It is clear that $\mathcal{P}^c_m$ is a finite set. It implies that
\begin{equation}\label{Eqn of mu(g) over norm for large smallest prime}
    \sum_{\substack{\partial(g)\geq 0\\d_-(g)>m}}\frac{\mu(g)}{\|g\|}=\left(\sum_{\partial(g)\geq 0}\frac{\mu(g)}{\|g\|}\right)\cdot\left(\prod_{P\in\mathcal{P}^c_m}\left(1-\frac{1}{\|P\|}\right)^{-1}\right)
    =0.
\end{equation}
Finally, we show the desired result by using \eqref{Key est for M(n, m)}, \eqref{Eqn of mu(g) over norm for large smallest prime} and summation by parts. In particular, let $y$ be fixed. To apply summation by parts, we put
\begin{equation}
    M_y(x)=\sum_{\substack{0\leq\partial(g)\leq x\\d_-(g)>y}}\mu(g).
\end{equation}
It follows that
\begin{align*}
R(x, y)=&-\sum_{\smat{\partial(g)>x\\d_{-}(g)>y}}\frac{\mu(g)}{\|g\|}
=-\int_x^\infty\frac{dM_y(t)}{q^{t}}\\
=&\frac{M_y(x)}{q^{x}}-\ln q\int_x^\infty\dfrac{M_y(t)}{q^{t}}\,dt\\
\ll&q^{(\eta-1)x}+\ln q\int_x^\infty\mathrm{exp}(-c\sqrt{t})\,dt\\
\ll&q^{(\eta-1+\varepsilon)x},
\end{align*}
where $\varepsilon$ is an arbitrary positive number. Hence, for $1\leq m\leq \log\log n$ we also obtain $R(n, m)\ll q^{(\eta-1+\varepsilon)n}$ as desired. In conclusion, we complete the proof of the lemma.
\end{proof}
Now we are ready to prove ``$\eqref{Eqn: analogue of main result in global function feild}\Leftrightarrow \eqref{Eqn: to write into partial sum}$". 
\begin{proof}[Proof of ``$\eqref{Eqn: analogue of main result in global function feild}\Leftrightarrow \eqref{Eqn: to write into partial sum}$"]\label{Proof of keylemma2}
Set $f(\infty)=0$ for convenience. First, we break up the partial sum of  \eqref{Eqn: to write into partial sum} into two sums: 
\begin{align*}
\sum_{\smat{0\le \partial(g)\le n\\ g\in \mathfrak{D}(\cG,S)}}\frac{\mu*a(g)}{\|g\|}f(d_{-}(g))
=&\sum_{\smat{0\le \partial(g)\le n\\ g\in \mathfrak{D}(\cG,S)}}\frac{f(d_{-}(g))}{\|g\|}\sum_{r|g}\mu(r)a(g-r)\\
=&\sum_{\smat{0\le \partial(g)\le n\\ g\in \mathfrak{D}(\cG,S)}}\frac{f(d_{-}(g))}{\|g\|}\mu(g)+\sum_{\smat{0\le \partial(g)\le n\\ g\in \mathfrak{D}(\cG,S)}}\frac{f(d_{-}(g))}{\|g\|}\sum_{\substack{r|g\\r\neq g}}\mu(r)a(g-r)\\
\colonequals&S_7+S_8.
\end{align*}
To show ``$\eqref{Eqn: analogue of main result in global function feild}\Leftrightarrow \eqref{Eqn: to write into partial sum}$", it suffices to show $S_8=o(1)$. Let $g=g'+P_{\min}(g)$. Then we have
\begin{align*}
    S_8=&\sum_{\smat{0\le \partial(g)\le n\\ g\in \mathfrak{D}(\cG,S)}}\frac{f(d_{-}(g))}{\|g\|}\sum_{\substack{r|g\\r\neq g}}\mu(r)a(g-r)\\
    =&\sum_{\substack{0\leq \partial(P_{\min}(g))\leq n\\ P_{\min}(g)\in S}}\frac{f(\partial(P_{\min}(g)))}{\|P_{\min}(g)\|}\sum_{\smat{0\le \partial(g^{\prime})\le n-\partial(P_{\min}(g))\\d_-(g^{\prime})>\partial(P_{\min}(g))}}\frac{1}{\|g^{\prime}\|}\sum_{\substack{r|g\\P_{\min}(g)| r \\r\neq g}}\mu(r)a(g^{\prime}+P_{\min}(g)-r)\\
    &+\sum_{\substack{0\leq \partial(P_{\min}(g))\leq n\\ P_{\min}(g)\in S}}\frac{f(\partial(P_{\min}(g)))}{\|P_{\min}(g)\|}\sum_{\smat{0\le \partial(g^{\prime})\le n-\partial(P_{\min}(g))\\d_-(g^{\prime})>\partial(P_{\min}(g))}}\frac{1}{\|g^{\prime}\|}\sum_{\substack{r|g'}}\mu(r)a(g^{\prime}+P_{\min}(g)-r)\\
    \colonequals&S_9+S_{10}.
\end{align*}
Next, we will show that both $S_9$ and $S_{10}$ are error terms of size $o(1)$.
\begin{enumerate}
    \item For $S_9$, let $r=P_{\min}(g)+r^{\prime}$ and $g^{\prime}=h^{\prime}+r^{\prime}$ and then exchange the order of the first and the second summations. Since the function $a$ only supports on $\mathfrak{D}(\cG, S)$, we have that
    \begin{align*}
        S_9=&-\sum_{\substack{0\leq \partial(P_{\min}(g))\leq n\\ P_{\min}(g)\in S}}\frac{f(\partial(P_{\min}(g)))}{\|P_{\min}(g)\|}\sum_{\substack{1\leq \partial(h^\prime)\leq n-\partial(P_{\min}(g))\\d_-(h^\prime)>\partial(P_{\min}(g))\\h^\prime\in\mathfrak{D}(\cG, S)}}\frac{a(h^\prime)}{\|h^\prime\|}\sum_{\substack{1\leq \partial(r^\prime)\leq n-\partial(P_{\min}(g))-\partial(h^\prime)\\d_-(r^\prime)>\partial(P_{\min}(g))}}\frac{\mu(r^\prime)}{\|r^\prime\|}\\
        =&-\sum_{\substack{1\leq \partial(h^\prime)\leq n\\h^\prime\in\mathfrak{D}(\cG, S)}}\frac{a(h^\prime)}{\|h^\prime\|}\sum_{\substack{1\leq \partial(P_{\min}(g))\leq n-\partial(h^\prime)\\\partial(P_{\min}(g))<d_-(h^\prime)\\ P_{\min}(g)\in S}}\frac{f(\partial(P_{\min}(g)))}{\|P_{\min}(g)\|}\sum_{\substack{1\leq \partial(r^\prime)\leq n-\partial(P_{\min}(g))-\partial(h^\prime)\\d_-(r^\prime)>\partial(P_{\min}(g))}}\frac{\mu(r^\prime)}{\|r^\prime\|}\\
          \ll&\sum_{\substack{1\leq \partial(h^\prime)\leq [\log\log n]\\h^\prime\in\mathfrak{D}(\cG, S)}}\frac{|a(h^\prime)|}{\|h^\prime\|}\sum_{\substack{1\leq \partial(P_{\min}(g))\leq n-\partial(h^\prime)\\\partial(P_{\min}(g))<d_-(h^\prime)\\ P_{\min}(g)\in S}}\frac{1}{\|P_{\min}(g)\|}|R(n-\partial(P_{\min}(g))-\partial(h^\prime), \partial(P_{\min}(g)))|\\
         +&\sum_{\substack{[\log\log n]+1\leq \partial(h^\prime)\leq n\\h^\prime\in\mathfrak{D}(\cG, S)}}\frac{|a(h^\prime)|}{\|h^\prime\|}\sum_{\substack{1\leq \partial(P_{\min}(g))\leq n-\partial(h^\prime)\\\partial(P_{\min}(g))<d_-(h^\prime)\\ P_{\min}(g)\in S}}\frac{1}{\|P_{\min}(g)\|}\\
          \ll&q^{-c(n-2\log\log n)}\sum_{\substack{1\leq \partial(h^\prime)\leq [\log\log n]\\h^\prime\in\mathfrak{D}(\cG, S)}}\frac{|a(h^\prime)|}{\|h^\prime\|}\log\log\|h^\prime\|
         +\sum_{\substack{[\log\log n]+1\leq \partial(h^\prime)\leq n\\h^\prime\in\mathfrak{D}(\cG, S)}}\frac{|a(h^\prime)|}{\|h^\prime\|}\log\log \|h^\prime\|\\
         =&O(q^{-c(n-2\log\log n)})+o(1)=o(1).
    \end{align*}
    Here the constant $c=-(\eta-1+\varepsilon)>0$ by Lemma~\ref{lemma 2 for R(x y)} and the last inequality follows from the assumption of $a$ as in Theorem~\ref{Thm: main_thm_convolution_Axiom A_sharp}. 
    \item Similarly for $S_{10}$, let $g^{\prime}=h+r$ and then we have that
\begin{align*}
   S_{10}=&\sum_{\substack{0\leq \partial(P_{\min}(g))\leq n\\ P_{\min}(g)\in S}}\frac{f(\partial(P_{\min}(g)))}{\|P_{\min}(g)\|}\sum_{\substack{1\leq \partial(h)\leq n-\partial(P_{\min}(g))\\d_-(h)>\partial(P_{\min}(g))\\ \text{or}\, h=0}}\frac{a(h+P_{\min}(g))}{\|h\|}\sum_{\substack{1\leq \partial(r)\leq n-\partial(P_{\min}(g))-\partial(h)\\d_-(r)>\partial(P_{\min}(g))}}\frac{\mu(r)}{\|r\|}\\
=&\sum_{0\leq \partial(h)\leq n}\sum_{\substack{0\leq \partial(P_{\min}(g))\leq n-\partial(h)\\\partial(P_{\min}(g))<d_-(h)\\ P_{\min}(g)\in S}}\frac{f(\partial(P_{\min}(g)))a(h+P_{\min}(g))}{\|P_{\min}(g)\|\|h\|}\sum_{\substack{1\leq \partial(r)\leq n-\partial(P_{\min}(g))-\partial(h)\\d_-(r)>\partial(P_{\min}(g))\\}}\frac{\mu(r)}{\|r\|}\\
\ll&\sum_{0\leq \partial(h)\leq [\log\log n]}\sum_{\substack{1\leq \partial(P_{\min}(g))\leq n-\partial(h)\\\partial(P_{\min}(g))<d_-(h)\\ P_{\min}(g)\in S}}\frac{|a(h+P_{\min}(g))|}{\|P_{\min}(g)\|\|h\|}|R(n-\partial(P_{\min}(g))-\partial(h),\partial(P_{\min}(g)))|\\
+&\sum_{\substack{[\log\log n]+2\leq \partial(g)\leq n\\g\in\mathfrak{D}(\cG, S)}}\frac{|a(g)|}{\|g\|}\\
\ll &q^{-c(n-2\log\log n)}\sum_{\substack{1\leq \partial(g)\leq n\\ g\in\mathfrak{D}(\cG, S)}}\frac{|a(g)|}{\|g\|}+\sum_{\substack{[\log\log n]+2\leq \partial(g)\leq n\\g\in\mathfrak{D}(\cG, S)}}\frac{|a(g)|}{\|g\|}\\
=&O(q^{-c(n-2\log\log n)})+o(1)=o(1).
\end{align*}
\end{enumerate}
Again, the constant $c=-(\eta-1+\varepsilon)>0$ by Lemma~\ref{lemma 2 for R(x y)} and the last inequality follows from the assumption of $a$ as in Theorem~\ref{Thm: main_thm_convolution_Axiom A_sharp}. Thus, we conclude that 
\begin{equation}\label{keyequation}
\sum_{\smat{0\le \partial(g)\le n\\ g\in \mathfrak{D}(\cG,S)}}\frac{\mu*a(g)}{\|g\|}f(d_{-}(g))=\sum_{\smat{0\le \partial(g)\leq n\\ g\in \mathfrak{D}(\cG, S)}}\frac{\mu(g)}{\|g\|}f(d_{-}(g))+o(1).
\end{equation}
\end{proof}
\subsection{Proof of Corollary~\ref{Main Corollary_Axiom A_sharp}}\label{Sect: pf_cor}
In this section, we prove Corollary~\ref{Main Corollary_Axiom A_sharp} in details. For the convenience of the readers, we cite it here.  

\begin{corollary}
Suppose that $a\colon \cG\to \C$ is an arithmetic function supported on $\mathfrak{D}(\cG, S)$ with $a(e_{\cG})=1$ and $|a(g)|\ll \|g\|^{-\alpha}$ for some $\alpha>0$. If $S\subseteq\cP$ has a natural density $\delta(S)$, then
\begin{equation}\label{mainthmeqvaphi_cited again}
		-\lim_{n\to\infty}  \sum_{\smat{1\le \partial(g)\le n\\ g\in \mathfrak{D}(\cG,S)}}\frac{\mu*a(g)}{\varphi(g)}=\delta(S).
	\end{equation}	
Here, $\varphi(g)$ is Euler's totient function defined by $\varphi(g)=\|g\|\prod_{P|g}\left(1-\frac{1}{\|P\|}\right).$
\end{corollary}
\begin{proof}
The proof is similar to the approach in \cite{Wang2020jnt}. In particular, we will apply Theorem~\ref{Thm: main_thm_convolution_Axiom A_sharp} to show \eqref{mainthmeqvaphi_cited again}.
Define
\begin{equation}
	b(g)\colonequals\sum_{h|g}\mu*a(h)\frac{\|h\|}{\varphi(h)}. 
	\end{equation}	
Then by the M\"obius inversion formula, we have	
\begin{equation}
\frac{\mu*a(g)}{\varphi(g)}=\frac{\mu*b(g)}{\|g\|}.
\end{equation}
Clearly, $b(e_{\cG})=1$. By Theorem~\ref{Thm: main_thm_convolution_Axiom A_sharp}, to prove \eqref{mainthmeqvaphi}, it suffices to show that
\begin{equation}\label{eqtoproveforcor12}
\sum_{\smat{0\le \partial(g)\le n\\ g\in \mathfrak{D}(\cG,S)}}\frac{|b(g)|}{\|g\|}\log\log \|g\|<\infty.
\end{equation}
By the definition of  $\mu*a(r)$ and go through a similar argument for elements in $\cG$ as in \cite{Wang2020jnt}, we have 
\begin{equation}
b(g)=\sum_{h|g, (h, g-h)=0}\frac{\|h\|}{\varphi(h)}a(h)\prod_{\smat{P|g-h}}\frac{1}{1-\|P\|}. 
\end{equation}
Furthermore, we have that 
\begin{equation}\label{eqofabsofbA}
|b(g)|\leq \sum_{h|g, (h, g-h)=0}|a(h)|\frac{\|h\|}{\varphi(h)}\prod_{\smat{P|g-h}}\frac{1}{\|P\|-1}. 
\end{equation}
On the other hand, by the standard elementary technique, for every $0<\alpha\leq 1$ one can show that
\begin{equation}
\frac{\|g\|^{(1-\alpha)}}{\varphi(g)}\to 0 \quad\mbox{as $\partial(g)\to\infty$}.
\end{equation}
In particular, we have 
\begin{equation}
\frac{\|g\|}{\varphi(g)}\ll \|g\|^{\frac{\alpha}{2}}.    
\end{equation}
Thus, by \eqref{eqofabsofbA} and the assumption that $|a(g)|\ll \|g\|^{-\alpha}$ with $\alpha>0$, we get the following estimate for $b(g)$:
\begin{equation}\label{eqofabsofbA2}
|b(g)|\leq \sum_{r|g}\|r\|^{-\frac{\alpha}{2}}\prod_{\smat{P|g-r}}\frac{1}{\|P\|-1}.
\end{equation}
Put
\begin{equation}\label{eqofabsofbA2cg}
c(g)=\sum_{r|g}\|r\|^{-\frac{\alpha}{2}}\prod_{\smat{P|g-r}}\frac{1}{\|P\|-1}.
\end{equation}
Then $c(g)$ is the Dirichlet convolution of $c_1(g)=\|g\|^{-\frac{\alpha}{2}}$ and $c_2(g)=\prod_{P|g}\frac{1}{\|P\|-1}$ . It is easy to see that 
\begin{equation}
    \sum_{\partial(g)=0}^\infty\frac{c_1(g)}{\|g\|^s}=\zeta_{\cG}(s+\frac{\alpha}{2}),  
\end{equation}
which is absolutely convergent on $\mathrm{Re}(s)>1-\alpha/2$ by the analytic property of $\zeta_{\cG}(s)$. Moreover, one can show that  
\begin{equation}\label{c2 Dirichlet is abs con}
    \sum_{\partial(g)=0}^\infty\frac{c_2(g)}{\|g\|^s} 
\end{equation} is absolutely convergent on $\sigma=\mathrm{Re}(s)>0$. Observing that $c_2(g)$ is multiplicative, we have
\begin{equation}
   \sum_{\partial(g)\leq N}\frac{|c_2(g)|}{\|g\|^{\sigma}}\leq \prod_{\partial(P)\leq N}\left(1+\sum_{n\ge 1}\frac{|c_2(nP)|}{\|P\|^{n\sigma}}\right).
\end{equation}
On the other hand, we have
\begin{align*}
\sum_{\partial(P) \ge 0}\sum_{n\ge 1}\frac{|c_2(nP)|}{\|P\|^{n\sigma}}=&\sum_{P}\frac{1}{(\|P\|-1)(\|P\|^\sigma-1)}\\
=&\sum_m\sum_{\partial(P)=m}\frac{1}{(q^{m}-1)(q^{m\sigma}-1)}\\
\ll &\sum_m\frac{1}{q^{m\sigma}}<\infty
\end{align*}
for $\sigma=\mathrm{Re}(s)>0$.
By the well known theorem of infinity product we conclude that \eqref{c2 Dirichlet is abs con} is absolutely convergent for $\mathrm{Re}(s)>0$.
It follows that on $\mathrm{Re}(s)>\sigma_0$, where $\sigma_0=\max\{1-\alpha/2, 0\}<1$, we have
\begin{equation}
    \sum_{\partial(g)=0}^\infty\frac{c(g)}{\|g\|^s}=\zeta_{\cG}(s+\frac{\alpha}{2})\sum_{\partial(g)=0}^\infty\frac{c_2(g)}{\|g\|^s}.  
\end{equation}
Therefore, the derivative of $\sum_{\partial(g)=0}^\infty c(g)\|g\|^{-s}$ is convergent at $s=1$, which implies that
\begin{equation}
\sum_{\partial(g)=0}^\infty \frac{c(g)}{\|g\|}\log \|g\|<\infty.    
\end{equation}
It follows immediately from \eqref{eqofabsofbA2} and \eqref{eqofabsofbA2cg} that
\begin{equation}\label{eqofabsofsumofbA}
    \sum_{\partial(g)=0}^\infty \frac{|b(g)|}{\|g\|}\log \|g\|<\infty.
\end{equation}
Hence, by \eqref{eqofabsofsumofbA} we obtain \eqref{eqtoproveforcor12} and so the desired result \eqref{mainthmeqvaphi} follows.

\end{proof}

\section{Proofs of main results for Axiom \texorpdfstring{$A$}{A} type arithmetical semigroups}\label{Sect: Free abe. semigp Axiom A}
In this section, we discuss the situation of Axiom $A$ type arithmetical semigroups. Similarly as in Sect.~\ref{Sect: Proofs for semigps Axiom A_sharp}, there is also a key result, Theorem~\ref{keylemma12combined_Axiom A}, and then Theorem~\ref{Thm: main_thm_convolution_Axiom A} follows as an immediate consequence. 
Throughout this section, we fix $\cG$ to be an Axiom $A$ type arithmetical semigroup, and adopt the definitions and notations in Sect.~\ref{Sect: Axiom_A}.

For any element $g\in \cG$, we take $\rN^+(g)\colonequals\max\set{ \norm{P}: P|g}$ to be the largest norm of all prime factors of $g$, and $\rN^+(e_{\cG})=1$. We also define the number $Q_S(g)$ in the same way for Axiom $A^{\#}$ type arithmetical semigroups. More precisely, we let
\begin{equation}\label{Defn of QSg for Axiom A}
    Q_S(g)\colonequals\#\set{P\in S: \norm{P}=\rN^+(g), P|g}
\end{equation}
be the number of prime factors of $g$ in $S$ attaining the maximal norm $\rN^+(g)$. Since most of proofs of our results in this section are very similar with those for the case of Axiom $A^{\#}$ type arithmetical semigroup, we will just describe them briefly and only focus on key steps. First, we have the following equidistribution result of the largest (in terms of norm $\|\cdot\|$) prime factors of elements in $\cG$. 
\begin{lemma}\label{equidistribution_AxiomA} 
    If $S\subseteq \cP$ has a natural density $\delta(S)$, then
    \begin{equation}
        \sum_{2\leq \norm{g}\leq x}Q_S(g)\sim c_{\cG}\,\delta(S)x.
    \end{equation}
\end{lemma}
\begin{proof}
The proof is similar as that for \cite[Theorem~3.1]{KuralMcDonaldSah2020}.
\end{proof}
Moreover, with an analogous argument as that for \cite[Lemma~2.1]{SweetingWoo2019} we have the following duality lemma for Axiom $A$ type arithmetical semigroups.
\begin{lemma}\label{duality_Axiom A}
	Suppose that $f:\N\to\C$ is an arithmetic function with $f(1)=0$. Then for any $g\in \cG$ we have
	\begin{equation}
    \sum_{h|g}\mu(h)1_{\mathfrak{D}(\cG, S)}(h)f(\rN_-(h))=-Q_S(g)f(\rN^+(g)).
\end{equation}
	Here $1_{\mathfrak{D}(\cG,S)}$ is the indicator function on $\mathfrak{D}(\cG,S)$ defined as in \eqref{Definition of distingushable_Axiom A}.
\end{lemma}
The following two lemmas is giving Axiom $A$ type arithmetical semigroups analogues of bounds for partial sums involving the M\"{o}bius function that we will need in our proofs.
\begin{lemma}\label{Lemma improves SW19 Lemma 2.2}
For some positive constant $c$ we have
\begin{align}
    \sum_{\norm{g}\leq x}\mu(g)=&O\big(x\,\mathrm{exp}\{-c(\log x)^{1/2}\}\big),\label{sum of mu leq T}\\
        \sum_{\norm{g}\leq x}\frac{\mu(g)}{\norm{g}}=&O\big(\mathrm{exp}\{-c(\log x)^{1/2}\}\big).\label{sum of mu/Norm leq T}
\end{align}
\end{lemma}
\begin{proof}
This directly generalizes \cite[Lemma~2.2]{SweetingWoo2019}, which indicates that the proof relies on the zero-free region for the zeta function $\zeta_{\cG}(s)$. In particular, the analogous zero-free region exists for $\zeta_{\cG}(s)$; see \cite[Theorem~1]{DiamondMontgomeryVorhauer2006}. Thus we can use the classical Perron-type argument as in \cite[Theorem~2.1]{FujisawaMinamide2018} to obtain the desired results.
\end{proof}
\begin{lemma}\label{Key Lemma for Number fields} 
For any $x, y\geq 1$, define
\begin{equation}
    M(x, y)\colonequals\sum_{\substack{\norm{g}\leq x\\\rN_-(g)>y}}\mu(g).
\end{equation}
There is a positive constant $c$, which depends only on $\cG$, such that

\begin{equation}\label{M(T, Y)_I}
    M(x, y)\ll x\,\mathrm{exp}(-c(\log x)^{1/2})\prod_{\norm{P}\leq y}(1-\norm{P}^{-1/2})^{-1},
\end{equation}
uniformly for $x, y\geq 1$. In particular, we have
\begin{equation}\label{M(T, Y)_II}
    M(x, y)\ll x\,\mathrm{exp}(-c(\log x)^{1/2}+y).
\end{equation}
\end{lemma}
\begin{proof}
By using the same zero-free region as in Lemma~\ref{Lemma improves SW19 Lemma 2.2} and the similar argument as in \cite[Lemma~4.2]{KuralMcDonaldSah2020}, the desired result follows. 
\end{proof}
To prove Theorem~\ref{keylemma12combined_Axiom A}, we will also need a more general version of Axer's theorem.
\begin{theorem}[Axer \cite{Axer1910}]\label{Axer Theorem}
If $f$ is a complex-valued function on $\cG$ satisfying
$$
\sum_{\norm{g}\leq x}|f(g)|=O(x),\qquad \sum_{\norm{g}\leq x}f(g)=o(x),
$$
then
$$
\sum_{\norm{g}\leq x}\left(c_{\cG}\frac{x}{\norm{g}}-\mathcal{N}\left(\frac{x}{\norm{g}}\right)\right)f(g)=o(x).
$$
\end{theorem}
Next, with a similar process as in the proof for \cite[Lemma~4.2]{KuralMcDonaldSah2020} by using Lemma~\ref{Key Lemma for Number fields}, we can get the following intermediate bound, which is needed to verify the conditions of Axer's theorem when we apply it to prove Theorem~\ref{keylemma12combined_Axiom A}.
\begin{lemma}\label{averagemu_axiomA}
	For any bounded function $f$,	we have
	\begin{equation}\label{averagemueq_axiomA}
		\sum_{\norm{g}\leq x}\mu(g)1_{\mathfrak{D}(\cG,S)}(g)f(\rN_{-}(g))=O_{f,\cG}\of{\frac{x}{\log\log x}}.
	\end{equation}
\end{lemma}
Now, we are ready to prove the following intermediate theorem.
\begin{theorem}\label{keylemma12combined_Axiom A}
    Let $f$ be any bounded arithmetic function with $f(1)=0$.     Let $a$ be an arithmetic function as in Theorem~\ref{Thm: main_thm_convolution_Axiom A}. The followings are equivalent:
    \begin{equation}\label{Eqn: equidis of largest prime factors_Axiom A}
        \sum_{\norm{g}\leq x}Q_S(g)f(\rN^+(g))\sim c_{\cG}\delta(S)x.
    \end{equation}
    \begin{equation}\label{Eqn: analogue of main result in global function feild_Axiom A}
        -\lim_{x\to\infty}\sum_{\substack{2\leq \norm{g}\leq x\\g\in\mathfrak{D}(\cG, S)}}\frac{\mu(g)f(\rN_-(g))}{\norm{g}}=\delta(S).
    \end{equation}
    \begin{equation}\label{Eqn: to write into partial sum_Axiom A}
        -\lim_{x\to\infty}\sum_{\substack{2\leq \norm{g}\leq x\\g\in\mathfrak{D}(\cG, S)}}\frac{\mu*a(g)f(\rN_-(g))}{\norm{g}}=\delta(S).
    \end{equation}    
\end{theorem}

\begin{proof}``$\eqref{Eqn: equidis of largest prime factors_Axiom A}\Leftrightarrow \eqref{Eqn: analogue of main result in global function feild_Axiom A}$": It is easy to see that the hypotheses of Theorem~\ref{Axer Theorem} are satisfied since the estimate \eqref{Eqn: Axiom_A} of $\mathcal{N}(x)$, Lemma~\ref{averagemu_axiomA} and the assumption that $f$ is bounded.
Then the desired result follows from Lemma~\ref{duality_Axiom A} and Theorem~\ref{Axer Theorem} by applying the analogous approach in \cite[Lemma~4.3]{KuralMcDonaldSah2020}. In fact, we have
\begin{align*}
  c_{\cG}\delta(S)x\sim \sum_{\norm{g}\leq x}Q_S(g)f(\rN^+(g))=&-\sum_{\norm{g}\leq x}\sum_{h|g}\mu(h)1_{\mathfrak{D}(\cG, S)}(h)f(\rN_-(h)) \\ 
  =&-\sum_{\norm{h}\leq x}\mu(h)1_{\mathfrak{D}(\cG, S)}(h)f(\rN_-(h))\sum_{\norm{g}\leq x/\norm{h}}1\\
  =&-c_{\cG}x\sum_{\substack{2\leq \norm{h}\leq x\\h\in\mathfrak{D}(\cG, S)}}\frac{\mu(h)f(\rN_-(h))}{\norm{h}}+o(x).
\end{align*}
``$\eqref{Eqn: analogue of main result in global function feild_Axiom A}\Leftrightarrow \eqref{Eqn: to write into partial sum_Axiom A}$": For any $x, y\geq 1$, we define
\begin{equation}
    R(x, y)\colonequals\sum_{\substack{1\leq\norm{g}\leq x\\\rN_-(g)>y}}\frac{\mu(g)}{\norm{g}}.
\end{equation}
By the analogous arguments as in the proof of Lemma~\ref{lemma 1 for abs. value. of R(x y)} we can obtain the same elementary bound for $R(x, y)$. That is, for any $x, y\geq 1$ we have
\begin{equation}\label{eqn_elementary bound for Rxy_axiom A}
    R(x, y)=O(1).
\end{equation}
Furthermore, for $1\leq y\leq \log\log x$, we can obtain the following refined estimate
\begin{equation}\label{eqn_refined bound for Rxy_axiom A}
    R(x, y)\ll \mathrm{exp}(-c(\log x)^{1/2}),
\end{equation}
where the positive constant $c$ depends only on $\cG$. In fact, for $y=1$, the estimate \eqref{eqn_refined bound for Rxy_axiom A} is just \eqref{sum of mu/Norm leq T} in Lemma~\ref{Lemma improves SW19 Lemma 2.2}. As for $1< y\leq \log\log x$, the estimate \eqref{eqn_refined bound for Rxy_axiom A} follows by using Lemma~\ref{Key Lemma for Number fields} and summation by parts. Finally, after going through a similar process as in the proof for ``$\eqref{Eqn: analogue of main result in global function feild}\Leftrightarrow \eqref{Eqn: to write into partial sum}$" (Sect.~\ref{Sect. for Proof of the second Key Theorem}) by using the estimates \eqref{eqn_elementary bound for Rxy_axiom A} and \eqref{eqn_refined bound for Rxy_axiom A}, the desired result follows.
\end{proof}

As a result, Theorem~\ref{Thm: main_thm_convolution_Axiom A} follows immediately from Lemma~\ref{equidistribution_AxiomA} and Theorem~\ref{keylemma12combined_Axiom A} by taking $f(n)=1$ for all $n\in \Z_{\geq 2}$. As for the proof of Corollary~\ref{Main Corollary_Axiom A}, it is similar to that of Corollary~\ref{Main Corollary_Axiom A_sharp}.

\section{Applications of main theorems}\label{Sect: application_details}
This section is devoted to give some necessary background in order to make sense of our applications stated in Sect.~\ref{Sect: Application_short}. Moreover, we also include some results those are not stated in Sect.~\ref{Sect: Application_short} due to the limit of space.

\subsection{Number field and its integral ideal semigroup}\label{Sect: applicatioin_num_field}
For references of the followings, one can refer to \cite[Chap.~\RNum{1}]{Neukirch-alg-NT}. Recall that a number field $K$ is defined to be a finite extension of the rational field $\Q$ and the ring of algebraic integers $\mathcal{O}_K$ of $K$ is a Dedekind domain. Thus every ideal $\mathfrak{a}$ in $\mathcal{O}_K$ admits a uniquely factorization into prime ideals up to order, i.e., we have 
$$
\mathfrak{a}=\mathfrak{p}_1^{a_1} \mathfrak{p}_2^{a_2}\cdots \mathfrak{p}_m^{a_m}
$$
with each $\mathfrak{p}_i$ a prime ideal and each $a_i> 0$. And we let $\mathcal{I}_K$ be the semigroup of integral ideals. Since every Dedekind domain has Krull dimension one, each nonzero prime ideal $\mathfrak{p}$ is also maximal. Thus the quotient $\mathcal{O}_K/\mathfrak{p}$ is an integral zero-dimensional ring, i.e., a field. In fact, $\mathcal{O}_K/\mathfrak{p}$ is a finite field since it can be considered as a finite extension of the finite field $\Z/(\Z\cap \mathfrak{p})$. Hence the \emph{norm} $\|\mathfrak{p}\|\colonequals[\mathcal{O}_k:\mathfrak{p}]=\#(\mathcal{O}_K/\mathfrak{p})$ is well-defined. An effective form of Landau's prime ideal theorem for algebraic number fields (see \cite{LagariasOdlyzko1977}) implies that
$$
\pi_K(x)\colonequals\#\{\mathfrak{p}: \|\mathfrak{p}\|\leq x\}=\frac{x}{\log x} + O_K(x\,\mathrm{exp}(-c_K\sqrt{\log x}))
$$
for some constant $c_K>0$ depending only on $K$. In particular, the set $\mathcal{P}$ of all prime ideals is countable. Moreover, it is well known that
$$
\mathcal{N}(x)\colonequals\#\{\mathfrak{a}: \|\mathfrak{a}\|\leq x\}=c_{K}x+O(x^{\eta})
$$
with $\eta=1-1/[K:\Q]$ (see \cite{MurtyvanOrder2007}). According to all of above facts, one can see that all the conditions of Axiom $A$ type arithmetical semigroup are satisfied. In conclusion, an immediate application of Theorem~\ref{Thm: main_thm_convolution_Axiom A} deduces Corollary~\ref{Cor: Alladi_number_field}, which is stated again here for the convenience of readers. 

\begin{corollary}
Let $S\subseteq \cP$ be a subset of prime ideals with natural density $\delta(S)$. For any arithmetic function $a: \mathcal{I}_K\to \C$ supported on $\mathfrak{D}(\mathcal{I}_K, S)$ with $a(\mathcal{O}_K)=1$, and 
$$\lim_{x\to\infty}\sum_{\smat{2\le \norm{\mathfrak{a}}\le x\\ \mathfrak{a}\in\mathfrak{D}(\cI_K,S)}}\frac{|a(\mathfrak{a})|}{\norm{\mathfrak{a}}}\log\log \norm{\mathfrak{a}}<\infty.$$
Then
	\begin{equation*}
		-\lim_{x\to\infty}  \sum_{\smat{2\le \norm{\mathfrak{a}}\le x\\ \mathfrak{a}\in \mathfrak{D}(\cI_K,S)}}\frac{\mu*a(\mathfrak{a})}{\norm{\mathfrak{a}}}=\delta(S).
	\end{equation*}	
In particular, we have
\begin{equation*}
	-\lim_{x\to\infty}  \sum_{\smat{2\le \norm{\mathfrak{a}}\le x\\ \mathfrak{a}\in \mathfrak{D}(\cI_K,S)}}\frac{\mu(\mathfrak{a})}{\norm{\mathfrak{a}}}=\delta(S).
	\end{equation*}
\end{corollary}

\subsection{Algebraic varieties and their \texorpdfstring{$0$}{0}-cycle semigroups}\label{Sect: application_0_cycles}
Fix $p$ to be a prime number and let $q=p^r$ be a prime power. Let $\F_q$ be the corresponding finite field. Let $X$ be a $d$-dimensional projective smooth variety defined over $\F_q$. A $0$-cycle $A$ of $X$ over $\F_q$ is a $0$-dimensional subvariety of $X$ which is also defined over $\F_q$. By the Zariski topology assigned to $X$, this means that the base change $A_{\overline{\F}_q}$ of $A$ to the algebraic closure of $\F_q$ is a finite set of geometric points which is stable under the action of Galois group $\Gal(\overline{\F}_q/\F_q)$. 

For any geometric point $\widetilde{P}\in X(\overline{\F}_q)$, the induced \emph{effective prime $0$-cycle} (or \emph{prime cycle} for simplicity) $P=P(\widetilde{P})$ is defined to be the finite sum of all the distinct Galois conjugations of $\widetilde{P}$ over the ground field $\F_q$, i.e., if $\widetilde{P}$ is defined over $\F_{q^n}$ but not over any its proper subfield, then 
$$P\colonequals\sum\limits_{\sigma\in \Gal(\F_{q^n}/\F_q)} {\widetilde{P}}^{\sigma}.$$ 
It follows immediately that $P(\widetilde{P})=P(\widetilde{P}^{\sigma})$ for any $\sigma\in \Gal(\overline{\F}_q/\F_q)$. In this subsection, we denote by $\mathcal{P}$ the set of prime cycles (of $X$). And we define $\mathcal{A}_X$ to be the semigroup of $\F_q$-effective $0$-cycles of $X$, i.e., it is the semigroup consisting of all the finite sum of prime formal sum of the form 
\begin{equation*}
    A=\sum_{i=1}^k a_{i} P_i=a_1P_1+a_2P_2+\cdots +a_kP_k, \quad \text{with }P_i\in \mathcal{P}, a_i\in \Z_{\geq 0}\text{ and }k\in \Z_{\geq 1}.
\end{equation*}
In particular, $0$ is contained in $\mathcal{A}_X$. For each prime cycle $P$, we define its \emph{degree} $\partial(P)$ to be the minimal positive integer $m$ such that $P$ splits over $\F_{q^{m}}$, i.e., $m$ is the minimal integer such that there exists $\widetilde{P}$ as above such that $P=P(\widetilde{P})$. Note that this definition is independent of the choice of $\widetilde{P}$. Then we denote by $\|P\|\colonequals (q^d)^{\partial(P)}$ and call it the \emph{norm} of $P$\footnote{Our definition of norm is different from the classical version $\|P\|=q^{\partial(P)}$, which is more familiar by algebraic geometer. The reason we take this change is such normalization will make it easier to apply our main results. If one would like to stick with the classical norm, then one needs to replace $\|A\|$ by $\|A\|^d$ in the statement of Corollary~\ref{Cor: Alladi_varieties} to get the correct result.}. For any effective $0$-cycle $A$, its degree is defined to be the summation of its prime cycles and its norm is the product of the corresponding norms.

Recall that the Hasse-Weil zeta function of $X$ is defined to be 
\begin{equation}\label{Eqn: Z_X_1}
    Z_X(T)\colonequals\exp\left(\sum_{n=1}^{\infty}\#X(\F_{q^n})\frac{T^n}{n}\right)=\prod_{n=1}^{\infty}(1-T^n)^{-\pi^{\#}_{X}(n)},
\end{equation}
where $\pi^{\#}(n)\colonequals\{A\in\cA_X: \partial(A)=n\}$ is the total number of $0$-cycles of degree $n$. By the Weil conjecture \cite{Weil-conjecture}, which has been settled by a series of works of Dwork \cite{Dowrk-weil-conj}, Grothendieck \cite{Grothendieck-weil-conj} and Deligne \cite{Deligne-weil-conj} we know that 
\begin{equation}\label{Eqn: Z_X_2}
    Z_X(T)=\frac{\prod\limits_{1\leq i\leq 2d-1, i \text{ odd}}F_i(T)}{(1-T)(1-q^d T)\prod\limits_{2\leq i\leq 2d-2, i \text{ even}}F_i(T)},
\end{equation}
where each $F_i(T)\in \Z[T]$ is a polynomial of degree $B_i$ equaling the $i$th Betti number of $X$. Moreover, for each $i$ the corresponding $F_i(T)=\prod_{j=1}^{B_i}(1-\alpha_{ij}T)$, where $\alpha_{ij}$ are algebraic integers of absolute value $q^{i/2}$. In particular, $Z_X(T)$ does not have a zero at $T=-q^{-1}$. 

Taking logarithm of both sides of \eqref{Eqn: Z_X_2}, using \eqref{Eqn: Z_X_1} to express $\#X(\F_{q^n})$ in terms of $\pi^{\#}(m)$ with $m| n$ and applying the M\"obius inversion formula, one can deduce that 
$$
\pi^{\#}_X(n)=\frac{(q^d)^{n}}{n}+O\left(\frac{q^{\frac{(2d-1)n}{2}}}{n}\right).
$$ 
It follows that the set $\cP$ of prime $0$-cycles is countable. Then  
$$
G^{\#}(n)\colonequals\#\{A\in\cA_X: \partial(A)=n\}=c_{X}(q^d)^{n}+O\left(q^{\frac{2d-1}{2}n}\right)
$$
follows immediately as argued in \cite{IM91}. Here the constant $c_X>0$ only depends on $X$. Thus, by all above facts, it follows that $\mathcal{A}_{X}$ an arithmetical semigroup satisfying Axiom $A^{\#}$. In particular, by Theorem~\ref{Thm: main_thm_convolution_Axiom A_sharp} we obtain Corollary~\ref{Cor: Alladi_varieties} as desired which is copied here for the convenience of readers. 

\begin{corollary}
Suppose that $S\subseteq \cP$ is a subset of prime $0$-cycles with natural density $\delta(S)$. For any arithmetic function $a: \cA_{X}\to \C$ supported on $\mathfrak{D}(\cA_X, S)$ with $a(0)=1$, and 
$$\lim_{n\to\infty}\sum_{\smat{1\le \partial(A)\le n\\ A\in \mathfrak{D}(\cA_X,S)}}\frac{|a(A)|}{\|A\|}\log\log \|A\|<\infty.$$
Then
	\begin{equation}
		-\lim_{n\to\infty}  \sum_{\smat{1\le \partial(A)\le n\\ A\in \mathfrak{D}(\cA_X,S)}}\frac{\mu*a(A)}{\|A\|}=\delta(S). 
	\end{equation}	
In particular, we have
\begin{equation}
		-\lim_{n\to\infty}  \sum_{\smat{1\le \partial(A)\le n\\ A\in \mathfrak{D}(\cA_{X},S)}}\frac{\mu(A)}{\|A\|}=\delta(S).
	\end{equation}
\end{corollary}

Moreover, we give two more concrete examples of Corollary~\ref{Cor: Alladi_varieties} and Corollary~\ref{Main Corollary_Axiom A_sharp} as follows. 
\begin{example}
Let $X=\mathbb{P}^1_{\F_q}$ be the projective line over $\F_q$. For a polynomial $F$ in the affine coordinate ring $\F_q[x]$, we take the \emph{effective divisor of zeros} of $F$ to be
$$(F)_0\colonequals\sum_{{\rm{ord}}_P(F)\geq 0} {\rm{ord}}_P(F)P.$$
Using the fact that $\F_q(x)$ has class number one, the map $\phi: F\mapsto (F)_0$ gives a bijection of the following sets. 
\begin{align*}
    \set{\text{monic polynomials $F\in \F_q[x]$}} & \Longleftrightarrow \set{\text{effective divisors $D$ not supported by $(\infty)$}}\\
    \set{\text{irreducible monic polynomials $F_P\in \F_q[x]$}} & \Longleftrightarrow \set{\text{prime divisors $P\neq (\infty)$}}
\end{align*}
If a monic polynomial $F$ admit a unique irreducible factor $P_{\min}(F)$ with multiplicity one over $\F_q$, which is the minimal with respect to degree, then we say that $F$ is \emph{distinguishable}. Under the above bijection $\phi$, if $\infty\notin S$, then the set $\mathfrak{D}(X, S)$ can be translated to
$$
\mathfrak{D}(q,S)\colonequals\{F\in \F_q[x] \text{ is monic }\colon F \text{ is distinguishable and } (P_{\min}(F))_0\in S \}.
$$
In this situation, if one takes $a$ to be the identity of the convolution ring of arithmetic functions over effective divisors of $\mathbb{P}_{\mathbb{F}_q}^1$ in Corollary~\ref{Main Corollary_Axiom A_sharp} and uses the fact that $\delta(\cP)=\delta(\cP-\{\infty\})$, then one immediately obtains the following corollary. 
\begin{corollary}
For two coprime monic polynomials $f, g\in \mathbb{F}_q[x]$, we have	
\begin{equation}
    -\lim_{n\to\infty}  \sum_{\smat{F \in \mathfrak{D}(q,\cP), 1\le \deg F\le n\\ p_{\min}(F)\equiv f\mod g}}\frac{\mu(F)}{\varphi(F)}=\frac1{\varphi(g)}.
\end{equation}
\end{corollary}
\end{example}

\begin{example}\label{eg: 3}
We can also apply our results to study the distribution of subspaces. Let $V$ be an $(\ell+1)$-dimensional vector space over $\F_q$. Denote by $G(k+1,\ell+1)$ (or $G(k+1, V)$) the set of $(k+1)$-dimensional linear subspaces (over $\overline{\F}_q$) of $V$. Using the Pl\"ucker embedding ${G}(k+1, \ell+1)\hookrightarrow\mathbb{P}(\bigwedge^{k+1} V)$  \cite[Lecture~6]{Harris-AG-1st-course}, we can realize $G(k+1,\ell+1)$  as a $(k+1)(\ell-k)$-dimensional smooth projective variety defined over $\F_q$, i.e.,  the Grassmannian variety, which is denoted by $\mathbb{G}(k,\mathbb{P}_{\F_q}^{\ell})$. By the same embedding, one can see that \cite[Proposition~1.7.2]{Stanley-Enum-comb-1}
$$
\#\mathbb{G}(k,\mathbb{P}_{\F_q}^{\ell})(\F_{q^N})=\frac{|\GL(\ell+1, \F_{q^N})|}{q^{N(k+1)(\ell-k)}|\GL(k+1, \F_{q^N})|\cdot|\GL(\ell-k, \F_{q^N})|}.
$$
From this, one can deduce the zeta function of $\mathbb{G}(k+1, \ell+1)$, which turns out to be of the form 
$$
Z_{\mathbb{G}(k+1,\ell+1)}(T)=\frac{1}{(1-T)(1-qT)^{B_2}(1-q^2T)^{B_4}\cdots(1-q^d T)},
$$
where $d=(k+1)(\ell-k)$ and $B_i$ is the Betti number of $\mathbb{G}(k,\mathbb{P}_{\F_q}^{\ell})$ for each $i$. 

Now let $X=\mathbb{P}_{\F_q}^{\ell}$. Note that its affine cone $\mathbb{A}_{\F_q}^{\ell+1}$ can be understood as a $(\ell+1)$-dimensional space over $\F_q$. For every $(k+1)$-dimensional subspace $\widetilde{H}$ of $\mathbb{A}_{\overline{\F}_q}^{\ell+1}=\mathbb{A}_{\F_q}^{\ell+1}\times {{\rm Spec}\,\overline{\F}_q}$, denote by ${H}\colonequals H(\widetilde{H})=\cup_{\sigma}\widetilde{H}^{\sigma}$, where $\widetilde{H}^{\sigma}$ runs over all the distinct Galois conjugates of $\widetilde{H}$. It is easy to see that each ${H}$ is a reduced variety defined over the ground field $\F_q$. Then define  
$$
\cP_{k,\ell,q}\colonequals\{{H}| \widetilde{H}\in G(k+1, \mathbb{A}_{\F_q}^{\ell+1})\}.
$$ 
Consider the semigroup 
$$
\cA_{k,\ell,q}=\{a_1{H}_1\cup a_2{H}_2\cup \cdots\cup a_s{H}_s| {H}\in \cP_{k,\ell, q}, a_i\in \Z_{\geq 0}\}
$$ generated by $\cP_{k,\ell,q}$ by taking finite union in the sense of schemes (i.e. count the multiplicities). In particular, if all $a_i=0$, then we let the union to the empty set $\emptyset$. One can check that the above setups give rise to a well-defined abelian free semigroup structure of $\cA_{k,\ell,q}$ with the identity element $\emptyset$ and the additional operation equaling to the union operation. 

Recall that $d=(k+1)(\ell-k)$, for each ${H}\in \cP_{k,\ell,q}$, we define $\deg {H}$ to be the number of geometric irreducible components of ${H}$ and take the norm $\norm{H}$ to be $q^{d\deg H}$. 

In order to estimate the corresponding $\pi^{\#}(n)$ in this case, identify $G(k+1, \mathbb{A}_{\F_q}^{\ell+1})$ with $\mathbb{G}(k,\mathbb{P}_{\F_q}^{\ell})$ as in the first paragraph of this example. One can verify that under this identification, there is a natural bijection between $\cA_{k,\ell,q}$ and $\cA_{\mathbb{G}(k,\mathbb{P}_{\F_q}^{\ell})}$ (i.e., the semigroup of $0$-cycles of $\mathbb{G}(k,\mathbb{P}_{\F_q}^{\ell})$) and also a bijection between   $\cP_{k,\ell,q}$ and $\cP_{\mathbb{G}(k,\mathbb{P}_{\F_q}^{\ell})}$. Moreover, it is not hard to see that this identification respects to the degree maps and norm maps. Hence we can deduce that 
$$\pi^{\#}(n)=\frac{q^{(k+1)(\ell-k) n}}{n}+O\left(\frac{q^{\eta{(k+1)(\ell-k) n}}}{n}\right)$$
and $G^{\#}(n)=c_{\cA_{k, \ell, q}}q^{(k+1)(\ell-k) n}+O(q^{\eta(k+1)(\ell-k) n})$ with suitable constants $c_{\cA_{k, \ell, q}}>0$ and $0\leq \eta<1$. In particular, $\cA_{k, \ell, q}$ is an Axiom $A^{\#}$ type arithmetical semigroup by definition. Thus, it follows from Corollary~\ref{Cor: Alladi_varieties} that we obtain
\begin{equation*}
		-\lim_{n\to\infty}  \sum_{\smat{1\le \partial(H)\le n\\ H\in \mathfrak{D}(\cA_{k,\ell,q},S)}}\frac{\mu(H)}{\|H\|}=\delta(S),
	\end{equation*}
as long as $S\subseteq \cP_{k,\ell,q}$ has natural density $\delta(S)$ and the definition of $\mathfrak{D}(\cA_{k,\ell,q},S)$ is an analogue to \eqref{Definition of distingushable_Axiom A_sharp}. 
\end{example}

\subsection{Finite graphs and the semigroups of closed paths}\label{Sect: application_graph_theory}
This subsection follows from \cite{Terras-zeta-fun-graph}, especially its chapter~2. The graphs in this section will always be finite, connected and undirected. The degree of a vertex is the number of the edges connecting this vertex. Moreover, we will assume that all the graphs in this section do not contain a degree-$1$ vertex. Given a graph $G$, we will denote by $V$ the vertex set of $G$ and by $E$ the edge set of $G$. 

In order to define the prime element and the corresponding semigroup. We orient the edges in $E$ and label them as follows: 
$$
e_1, e_2, \cdots, e_m, e_{m+1}=e_1^{-1}, \cdots, e_{2m}=e_m^{-1},
$$
where $m=\#(E)$ is the number of unoriented edges and $e^{-1}_j=e_{j+m}$ is the edge $e_j$ with the opposite orientation. One oriented edge (i.e., the labeled edge) $a_j$ is called to \emph{follow} another oriented edge $a_i$ if the start vertex of $a_j$ (with respect to the orientation of $a_j$) is the end vertex of $a_i$. 

A \emph{path} $C$ is a sequence $\{a_1,a_2, \cdots, a_s\}$ of oriented edges of $G$ such that $a_{i+1}$ follows $a_i$. To simplify the notation, we will write such a path as $C=a_1\cdots a_s$. A \emph{closed path} is a path whose starting vertex and the terminal vertex coincide. In the following, we assume that all paths are closed and do not have a backtrack or tail, i.e., $a_{i+1}\neq a_i^{-1}$ for all $i=1, \ldots, s-1$ and $a_s\neq a_1^{-1}$. A (closed) path $C$ is called \emph{primitive} or \emph{prime} if $C\neq D^k$ for some positive integer $k>1$, i.e., we cannot find another path $D=a_1\cdots a_s$ such that $C=D\cdots D=(a_1\cdots a_s)  \cdots (a_1\cdots a_s)$.  

For a (closed) path $C=a_1\cdots a_s$, the \emph{equivalence class} $[C]$ means the following
$$
[C]=\{a_1\cdots a_s, a_2a_3\cdots a_sa_1, \cdots, a_sa_1\cdots a_{s-1}\}.
$$
That is, $[C]$ represents the set of paths which are only differed by the starting (hence also by the ending) vertex. We say that $[C]$ is a \emph{prime} if so is $C$. And we denote by $\nu(C)=\nu(a_1\cdots a_s)=s$ to be the \emph{length} of $C$. We also note that $[C]=[a_1\cdots a_s]$ and $[C^{-1}]=[a_s\cdots a_1]$ are considered to be distinct classes. 

One can verify that with above setups, every path $C$ can be decomposed into primitive paths and this decomposition induces a unique decomposition of $[C]$ into prime classes up to order. Thus we will consider the semigroup $\cA_G$ of all the path classes regarding to $G$ and it is clear that 
$$
\cA_G=\left\{[C]=\sum_{\text{finite sum}}c_i[P_i]\colon  [P_i]\in\cP \text{ and } c_i\in\Z_{\geq 0} \text{ for all } i\right\}.
$$
Here $\cP$ the set of all primes $[P]$ in $G$. In particular, $\cA_G$ is abelian since the equivalence classes of path do not care the starting vertex. 

\begin{definition}\cite[Definition~2.2]{Terras-zeta-fun-graph}
    The \emph{Ihara zeta function} $\zeta_G$ for a finite connected graph without degree-$1$ vertices is a complex function is defined by 
    $$
    \zeta_G(z)\colonequals\prod_{[P]\in\cP}\left(1-z^{\nu(P)}\right)^{-1}.
    $$
\end{definition}

The Ihara theorem generalized by Bass, Hashimoto, etc \cite[Theorem~2.5]{Terras-zeta-fun-graph} describes the properties of the Ihara zeta function. To state this result, we need to introduce some matrices associated to a graph. For a graph $G$ as above with $\ell=\#(V)$, the \emph{adjacency matrix} $A$ of $G$ is an $\ell\times \ell$ matrix with $(i,j)$th entry 
$$
a_{ij}=\begin{cases}
    \text{number of undirected edges connecting vertex $i$ to vertex $j$}, & \text{if }i\neq j,\\
    \text{$2\times$ number of loops at vertex $i$}, &\text{if }i=j.
\end{cases}
$$
Also, we associate $G$ with another diagonal matrix $Q$ whose $j$th diagonal entry $q_j$ such that $q_j+1$ is the degree of the $j$th vertex of $G$. 

\begin{theorem}[\cite{Terras-zeta-fun-graph}, Theorem~2.5]Let the graph $G$, the adjacency matrix $A$ and the diagonal matrix $Q$ be defined as above. And let $r=\#(E)-\#(V)+1$. Then 
    \begin{equation}\label{Eqn: Ihara_zeta_fun}
        \zeta_G(z)=\frac{1}{(1-z^2)^{r-1}\det(I-Az+Qz^2)}.
    \end{equation}
\end{theorem}
Since $G$ is connected, we know that $r-1\geq 0$ (in fact $r-1=0$ if and only if $G$ is a tree). It follows that $\zeta_G(z)$ does not have a zero. Moreover, let $R_G$ be the radius of convergence of $\zeta_G$, i.e., the $R_G$ is the minimum of the absolute value of all poles of $\zeta_G$. Then we have the following result of Kotani and Sunada.
\begin{theorem}\cite[Theorem~8.1]{Terras-zeta-fun-graph}
Let $G$ be as above, let $\alpha$ and  $\beta$ be two integers such that $\alpha+1$ is the maximal degree among all vertices and $\beta+1$ be the minimal degree, respectively. Then every pole $u$ of $\zeta_G$ satisfies 
$$
R_G\leq |u|\leq 1\quad\text{and}\quad \frac{1}{\alpha}\leq R_G\leq \frac{1}{\beta}.
$$
\end{theorem}
From this theorem, one can deduce (applying the same arguments as in the previous section) the corresponding prime number theorem of $\cA_G$. Indeed, if we take $\Delta_G\colonequals\gcd\{\nu(P)\colon [P] \text{ prime of }G\}$, then (see \cite[$\S$~2.7]{Terras-zeta-fun-graph})
$$
\pi^{\#}_G(n)\colonequals\#\{\text{primes }[P]\colon \nu(P)=\Delta_G n\}=\frac{1}{nR_G^{\Delta_G n}}+O\left(\frac{1}{n R_G^{\Delta_G \eta n}}\right)
$$
for some $0\leq \eta<1$. Thus the set $\cP$ of primes $[P]$ is countable, and similarly as previous sections, we have 
$$
G^{\#}(n)\colonequals\#\{[C]: \nu(C)=\Delta_G n\}=\frac{c_G}{R_G^{\Delta_G n}}+O\left(\frac{1}{R_G^{\Delta_G \eta n}}\right),
$$
where the positive constant $c_G$ only depends on $G$. By all above facts, we see that the semigroup $\cA_G$ satisfies all the conditions of Axiom $A^{\#}$ in Definition~\ref{Def: Axiom_A_sharp}, hence Corollary~\ref{Cor: Alladi_graph} follows as a consequence of Theorem~\ref{Thm: main_thm_convolution_Axiom A_sharp}. As what we did in other subsections, we restate the corollary here, in which we set the degree map $\partial([C])=\nu([C])$ for consistency and $\norm{[C]}=(1/R_G)^{\nu([C])}$. Moreover, the definition of $\mathfrak{D}(\cA_{G},S)$ is an analogue to \eqref{Definition of distingushable_Axiom A_sharp}. 

\begin{corollary}
Assume that $S\subseteq \cP$ is a subset of primitive path classes with natural density $\delta(S)$. For any arithmetic function $a: \cA_{G}\to \C$ supported on $\mathfrak{D}(\cA_G, S)$ with $a([0])=1$, and 
$$\lim_{n\to\infty}\sum_{\smat{1\le \partial([C])\le n\\ [C]\in \mathfrak{D}(\cA_G,S)}}\frac{|a([C])|}{\|[C]\|}\log\log \|[C]\|<\infty.$$
Then
	\begin{equation*}
		-\lim_{n\to\infty}  \sum_{\smat{1\le \partial([C])\le n\\ [C]\in \mathfrak{D}(\cA_G,S)}}\frac{\mu*a([C])}{\|[C]\|}=\delta(S). 
	\end{equation*}	
In particular, we have
\begin{equation*}
		-\lim_{n\to\infty}  \sum_{\smat{1\le \partial([C])\le n\\ [C]\in \mathfrak{D}(\cA_{G},S)}}\frac{\mu([C])}{\|[C]\|}=\delta(S).
	\end{equation*}
\end{corollary}

\section*{Acknowledgments}
We would like to thank Jeff Achter, Dongchun Han, Rachel Pries, Frank Thorne and Biao Wang for their helpful comments. We would also like to thank the referee for the detailed comments and suggestions.

	\bibliographystyle{plain}
	\bibliography{Generalizations_of_Alladi_s_formula_for_arithmetical_semigroups}

\end{document}